\documentclass[a4paper,12pt]{article}

\usepackage{amsfonts}
\usepackage{amssymb}
\usepackage{latexsym}
\usepackage{amsmath}
\usepackage{amsthm}

\usepackage{mathrsfs}

\usepackage[dvipsnames]{xcolor}

\usepackage[active]{srcltx}

\allowdisplaybreaks

\newtheorem{thm}{Theorem}[section]
\newtheorem{lem}[thm]{Lemma}
\newtheorem{prop}[thm]{Proposition}
\newtheorem{cor}[thm]{Corollary}

\newtheorem{rem}[thm]{Remark}

\newcommand {\B}    {\mathbb{B}}

\newcommand {\N}    {\mathbb{N}}

\newcommand {\esssp}  {\hbox{\rm ess-sp}\,}

\newcommand {\Ind}    {\hbox{\rm Ind}\,}

\begin{document}
\date{\today}
\title{Algebras of Toeplitz operators on the \\ $n$-dimensional unit ball
\thanks{This work was partially supported by CONACYT Project 238630, M\'exico and by DFG  (Deutsche Forschungsgemeinschaft), Project BA 3793/4-1.}}
\author{\normalsize Wolfram Bauer\\ 
\normalsize Institut f\"{u}r Analysis, Leibniz Universit\"{a}t \\ 
\normalsize  30167, Hannover, Germany\\
\normalsize {\sf bauer@math.uni-hannover.de} \\ \vspace{1mm}\\
\normalsize Raffael Hagger\\ 
\normalsize Institut f\"{u}r Analysis, Leibniz Universit\"{a}t \\ 
\normalsize  30167, Hannover, Germany\\
\normalsize {\sf raffael.hagger@math.uni-hannover.de} \\ \vspace{1mm}\\
\normalsize Nikolai Vasilevski \\
\normalsize Departamento de Matem\'aticas, CINVESTAV \\
\normalsize Apartado Postal 14-740, 07000, M\'exico, D.F., M\'exico \\
\normalsize {\sf nvasilev@math.cinvestav.mx} }
\maketitle
\begin{abstract}
We study $C^*$-algebras generated by Toeplitz operators acting on the standard weighted Bergman space $\mathcal{A}_{\lambda}^2(\mathbb{B}^n)$ over the unit ball $\mathbb{B}^n$ in $\mathbb{C}^n$. 
The symbols $f_{ac}$ of generating operators are assumed to be of a certain product type, see (\ref{Introduction_form_of_the_symbol}). By choosing $a$ and $c$ in different function algebras $\mathcal{S}_a$ 
and $\mathcal{S}_c$  over lower dimensional unit balls $\mathbb{B}^{\ell}$ and $\mathbb{B}^{n-\ell}$, respectively,  and by assuming the invariance of $a\in \mathcal{S}_a$ under some torus action we obtain 
$C^*$-algebras $\boldsymbol{\mathcal{T}}_{\lambda}(\mathcal{S}_a, \mathcal{S}_c)$ whose structural properties can be described. In the case of $k$-quasi-radial functions $\mathcal{S}_a$ and bounded uniformly continuous 
or vanishing oscillation symbols $\mathcal{S}_c$ we describe the structure of elements from  the algebra $\boldsymbol{\mathcal{T}}_{\lambda}(\mathcal{S}_a, \mathcal{S}_c)$, derive a list of irreducible 
representations of $\boldsymbol{\mathcal{T}}_{\lambda}(\mathcal{S}_a, \mathcal{S}_c)$, and prove completeness of this list in some cases. Some of these representations originate from a ``quantization effect'', induced 
by the representation of $\mathcal{A}_{\lambda}^2(\mathbb{B}^n)$ as the direct sum of Bergman spaces over a lower dimensional unit ball with growing weight parameter. As an application we derive the 
essential spectrum and index formulas for matrix-valued operators. 
\vspace{1mm}
{\bf keywords:} weighted Bergman spaces, operator $C^*$-algebra, irreducible representations
\end{abstract}

\section{Introduction}
The analysis of Banach  and $C^*$-algebras generated by Toeplitz operators acting on Bergman or Hardy spaces over a domain in $\mathbb{C}^n$ 
has a long history. Until today it is an area of active research (see \cite{AA,BV1,BV1-II,BV2,C,Le_2010, MiSuWi,XZ} and the references therein). The methods combine tools from functional analysis,  
complex analysis, algebra and, in the case of the present paper, they include recent asymptotic relations in deformation quantization in the sense of Rieffel \cite{BaHaVa, BV2}. 
In order to obtain manageable algebras one often starts with a generating family of Toeplitz operators $\mathcal{T}= \{T_f \: : \: f \in \mathcal{S}\}$, where $\mathcal{S}$ denotes a set of essentially bounded symbols with additional properties. By manageable algebras we mean operator $C^*$-algebras for which we can  
effectively describe their irreducible representations or the compact of their  
maximal ideals and the Gelfand transform, in the commutative case. For example, one may require some specific regularity or oscillatory behavior of elements in $\mathcal{S}$, or the invariance of $f \in \mathcal{S}$ under a Lie group action. If the resulting Banach algebras are commutative (e.g. see \cite{BV1,BV1-II, BV2, V1} for such cases), an explicit description of the maximal ideal space and the Gelfand transform provides some structural insight and, in particular, can be applied to calculate the spectrum of its elements. In the non-commutative case a complete classification of the irreducible representations of the algebra is of interest.  
\vspace{0.5ex}\par 
The present paper extends the results in \cite{BV2}. We study $C^*$-algebras generated by Toeplitz operators acting on the standard weighted Bergman spaces $\mathcal{A}_{\lambda}^2(\mathbb{B}^n)$ 
over the Euclidean unit ball $\mathbb{B}^n$ in $\mathbb{C}^n$. We represent $\mathcal{A}_{\lambda}^2(\mathbb{B}^n)$ as an infinite direct sum of Hilbert subspaces. 
Via a splitting of coordinates $z=(z^{\prime}, z^{\prime \prime}) \in \mathbb{C}^{\ell} \times \mathbb{C}^{n-\ell}$, we consider operator symbols of the form 
\begin{equation}\label{Introduction_form_of_the_symbol}
f_{ac}(z)  = a\left(\frac{z^{\prime}}{\sqrt{1-|z^{\prime\prime}|^2}}\right) c(z^{\prime\prime})\in L^{\infty}(\mathbb{B}^n), 
\end{equation}
with $a \in L^{\infty}(\mathbb{B}^{\ell})$ and $c \in L^{\infty} (\mathbb{B}^{n-\ell})$, respectively. If, in addition, $a$ is invariant under the diagonal action of a torus, then the above subspaces are invariant under the action of the 
Toeplitz operator $\boldsymbol{T}^{\lambda}_{f_{ac}}$ (after conjugation with a unitary operator). The restriction of $\boldsymbol{T}^{\lambda}_{f_{ac}}$ to each level 
acts as a tensor product of Toeplitz operators with symbols $a$ and $c$, respectively. These operators are  considered on differently weighted Bergman spaces over lower dimensional unit balls.  
Roughly speaking, we reduce the dimensionality of the problem. 
  
We then choose the functions $a$ and $c$ in specific symbol classes $\mathcal{S}_a$ and $\mathcal{S}_c$ and consider the $C^*$-algebra $\boldsymbol{\mathcal{T}}^{\lambda}(\mathcal{S}_a, \mathcal{S}_c)$ 
generated by all Toeplitz operators in $\{ \boldsymbol{T}_{f_{ac}}^{\lambda} \: : \: a \in \mathcal{S}_a, \: c \in \mathcal{S}_c\}$. In the paper \cite{BV2} we treated the lowest dimensional situation 
$n=2$ with $\mathcal{S}_c= C(\overline{\mathbb{D}})$ and $\mathcal{S}_a= L^{\infty}(0,1)$. The present work extends these results to arbitrary dimensions (which is rather straightforward)  and 
$k$-quasi-radial functions $\mathcal{S}_a=L_{k\textup{-}qr}^{\infty}(\mathbb{B}^{n-\ell})$. As for $\mathcal{S}_c$, we consider bounded uniformly continuous symbols  with respect to the Bergman metric 
and its subclass of symbols having vanishing oscillation at the boundary (which is significantly larger than $C(\overline{\mathbb{B}}^{n-\ell})$). We describe irreducible representations 
of the algebras $\boldsymbol{\mathcal{T}}^{\lambda}(\mathcal{S}_a, \mathcal{S}_c)$ and prove completeness of our list in different cases. Some of the representations arise via a ``quantization effect'' linked 
to the observation that in the above mentioned infinite orthogonal decomposition of  $\mathcal{A}_{\lambda}^2(\mathbb{B}^n)$ Bergman spaces with arbitrary large weight parameter appear. We characterize 
the quotient of the algebras by the ideal of compact operators. As as result the essential spectrum of elements in $\boldsymbol{\mathcal{T}}^{\lambda}(\mathcal{S}_a, \mathcal{S}_c)$  and an index formula 
in the case of matrix-valued Fredholm operators are derived. 
\vspace{0.5ex}\par
 The paper is organized as follows. In Section \ref{Section_1} we describe an isometric isomorphism between $\mathcal{A}_{\lambda}^2(\mathbb{B}^n)$ and an infinite orthogonal sum of Hilbert space tensor 
products involving differently weighted Bergman spaces over lower dimensional balls. \par 
In Section \ref{Section_3} we consider Toeplitz operators ${\bf T}_{f_{ac}}^{\lambda}$ acting on $\mathcal{A}_{\lambda}^2(\mathbb{B}^n)$ with symbols $f_{ac}$ of the form (\ref{Introduction_form_of_the_symbol}).  
The main observations are Proposition \ref{propositoin_decomposition_TO_form_3_3} and its Corollary \ref{co:f_ac} which state that under an invariance property of the operator 
symbol $a$ under a torus action and, up to unitary equivalence, ${\bf T}_{f_{ac}}^{\lambda}$ decomposes into an infinite sum of tensor products of operators. 
Each factor is a Toeplitz operator again acting on a (differently) weighted Bergman space over a lower dimensional unit ball. 
\par
Section \ref{Section_4} analyzes the $C^*$-algebra $\boldsymbol{\mathcal{T}}_{\lambda}(1,{\rm BUC}(\mathbb{B}^{n-\ell}))$, generated by Toeplitz operators ${\bf T}_{f_{c}}^{\lambda}$ with symbols $c$ in $\textup{BUC}(\mathbb{B}^{n-\ell})$, the bounded uniformly continuous functions  on  $\mathbb{B}^{n-\ell}$ with  respect to the Bergman metric. Theorem \ref{th:rep} provides a unique infinite sum representation for the elements of this algebra. This representation is based on the decomposition of the Bergman space in Section \ref{Section_1}. We compare and identify different $C^*$-algebras that naturally appear in the construction. As a main result we construct  families of irreducible representations of the algebra. A class of one-dimensional representations arises via a ``quantization effect'', i.e.~when the weight parameter tends to infinity. We note that the latter representations do not appear for the $C^*$-algebra generated by Toeplitz operators with uniformly continuous symbols on the ball $\mathbb{B}^n$ of full dimension. 
\par 
In Section \ref{se:VO} we consider the $C^*$-algebra  generated by Toeplitz operators ${\bf T}_{f_{c}}^{\lambda}$ having symbols $c$ in the space $\textup{VO}_{\partial}(\mathbb{B}^{n-\ell})$ 
of functions having vanishing oscillation at the boundary of $\mathbb{B}^{n-\ell}$. Since elements in $\textup{VO}_{\partial}(\mathbb{B}^{n-\ell})$ are bounded and uniformly continuous 
we obtain a subalgebra of   $\boldsymbol{\mathcal{T}}_{\lambda}(1,{\rm BUC}(\mathbb{B}^{n-\ell}))$. As an additional feature, we use the compactness of the semi-commutator for two Toeplitz operators with symbols 
in $\textup{VO}_{\partial}(\mathbb{B}^{n-\ell})$. The restriction of irreducible representations in Section  \ref{Section_4} to this subalgebra is shown to be a complete list of all such representations 
(cf. Theorem \ref{th:rep_list}). In the case of matrix-valued symbols of vanishing oscillation and based on recent results in \cite{XZ} we discuss the Fredholm property and index in Corollary \ref{Corollary_Fredholm_property_and_index}. 
\par 
Section \ref{Section_6} is concerned with the more complicated case of the $C^*$-algebra
\[\boldsymbol{\mathcal{T}}_{\lambda}(L^{\infty}_{k\textup{-}qr}, {\rm VO}_{\partial}(\mathbb{B}^{n-\ell}))\]
which is generated by Toeplitz operators ${\bf T}_{f_{ac}}^{\lambda}$ with $k$-quasi-radial symbols $a$ on $\mathbb{B}^{\ell}$ and symbols $c\in \textup{VO}_{\partial}(\mathbb{B}^{n-\ell})$.  In the spirit of 
Theorem \ref{th:rep_VO} we first derive a representation of the elements $\boldsymbol{T}^{\lambda}$ in this algebra in form of an infinite sum of tensor products. Differently from Section 
\ref{se:VO} this representation is not uniquely determined by $\boldsymbol{T}^{\lambda}$. Up to elements in a small ideal  we can assign to $\boldsymbol{T}^{\lambda}$  a ''symbol'' in 
${\rm SO}(\mathbb{Z}_+^m) \otimes {\rm VO}_{\partial}(\mathbb{B}^{n-\ell})$. Here ${\rm SO}(\mathbb{Z}_+^m)$ denotes a space of sequences with slow oscillation at infinity  (in a certain sense). 
A complete list of irreducible representations of  the algebra $\boldsymbol{\mathcal{T}}_{\lambda}(L^{\infty}_{k\textup{-}qr}, {\rm VO}_{\partial}(\mathbb{B}^{n-\ell}))$ is given in 
Theorem \ref{Section_6_complete_list_of_representations}. 
Finally we express the essential spectrum, characterize Fredholmness and provide an index formula for matrix-valued 
operators in $ \boldsymbol{\mathcal{T}}_{\lambda}(L^{\infty}_{k\textup{-}qr},{\rm VO}_{\partial}(\mathbb{B}^{n-\ell})) \otimes \mathrm{Mat}_p(\mathbb{C})$. 
\section{Bergman space representation}
\label{Section_1}
\setcounter{equation}{0}
Consider the open unit ball in $\mathbb{C}^n$: 
$$\mathbb{B}^n:=\big{\{}z=(z_1, \cdots, z_n)\in \mathbb{C}^n \: : \: |z|^2=|z_1|^2+\cdots+|z_n|^2<1\big{\}}.$$ 
For $\lambda >-1$ introduce the weighted Bergman space 
\begin{equation}\label{Def_weighted_Bergman_space}
\mathcal{A}_{\lambda}^2(\mathbb{B}^n):= \big{\{} f : \mathbb{B}^n \rightarrow \mathbb{C} \: : \: f \ \text{holomorphic and } v_{\lambda} \text{-square integrable} \big{\}},
\end{equation}
where the weighted measure $v_{\lambda}$ is absolutely continuous with respect to the Lebesgue volume form $dv(z)$ on $\mathbb{B}^n$ and given by: 
\begin{equation*}
dv_{\lambda}(z)=c_{\lambda}^{(n)} (1-|z|^2)^{\lambda} dv(z) \hspace{3ex} \text{ with} \hspace{3ex} c_{\lambda}^{(n)}:= \frac{\Gamma(n+\lambda+1)}{\pi^n \Gamma(\lambda+1)}. 
\end{equation*}
The inner product on $L^2(\mathbb{B}^n, dv_{\lambda})$ or $\mathcal{A}_{\lambda}^2(\mathbb{B}^n)$ is denoted by $\langle \cdot, \cdot \rangle_{\lambda}$. 
Let $\mathbb{Z}_+=\{0,1,2, \cdots \}$, and recall that the standard orthonormal basis of (\ref{Def_weighted_Bergman_space}) has the form: 
\begin{equation*}
\mathcal{E}:= [e^{\lambda}_{\alpha}(z) \: : \: \alpha \in \mathbb{Z}_+^n]\hspace{4ex} \text{with} \hspace{4ex} 
e^{\lambda}_{\alpha}(z)= \sqrt{\frac{\Gamma(n+|\alpha|+\lambda+1)}{\alpha! \Gamma(n+\lambda+1)}}z^{\alpha}. 
\end{equation*}
\par 
In what follows, we work with a certain orthogonal decomposition of the weighted Bergman space into an infinite orthogonal sum of Hilbert subspaces. In more detail:

With $\ell \in \{1, \cdots, n-1\}$ we divide the coordinates $(z_1, \cdots, z_n) \in \mathbb{C}^n$ into two parts: 
\begin{equation}\label{GL_separation_of_variables}
z=(\underbrace{z_1, \cdots, z_{\ell}}_{=:z^{\prime}}, \underbrace{z_{\ell +1}, \cdots ,z_n}_{=:z^{\prime \prime}})=(z^{\prime}, z^{\prime \prime}) \in \mathbb{C}^{\ell} \times \mathbb{C}^{n-\ell}.
\end{equation}
Let $m \leq \ell$, and fix a multi-index $k =(k_1, \cdots, k_m) \in \mathbb{N}^m$ with $|k|=k_1+\cdots +k_m=\ell$. We further divide $z^{\prime}\in \mathbb{C}^{\ell}$ into 
$m$ groups of coordinates as follows: 
\begin{equation}\label{definition_z_prime}
z^{\prime}=\big{(}z_{(1)}, \cdots, z_{(m)}\big{)}, \hspace{2ex} \text{where} \hspace{2ex} z_{(j)}
=\big{(} z_{k_1+\cdots +k_{j-1}+1}, \cdots , z_{k_1+ \cdots +k_j}\big{)}\in \mathbb{C}^{k_j},
\end{equation}
where $k_0 := 0$. For each $\rho=(\rho_1, \cdots, \rho_m)\in \mathbb{Z}_+^m$ we introduce the Hilbert space
\begin{equation*}
H_{\rho}:= \overline{\textup{span}}\left\{ e_{\alpha}^{\lambda} \, : \, \alpha=(\alpha_{(1)}, \cdots, \alpha_{(m)}, \alpha^{\prime \prime}) 
\in \mathbb{Z}_+^n \ \ \text{and} \ \ |\alpha_{(j)}|=\rho_j, \ \ j =1,...,m  \right\}.
\end{equation*}
Then we have the following orthogonal decomposition 
\begin{equation}\label{Orthogonal_decomposition_of_the_Bergman_space}
\mathcal{A}_{\lambda}^2(\mathbb{B}^n)=\bigoplus_{\rho \in \mathbb{Z}_+^m} H_{\rho}. 
\end{equation}
A similar orthogonal decomposition can be done for the Bergman space $\mathcal{A}^2_{\lambda}(\mathbb{B}^{\ell})$, whose orthogonal basis consists of the monomials  
\begin{equation*}
 e^{\lambda}_{\alpha^{\prime}}(w) = \sqrt{\frac{\Gamma(\ell+|\alpha^{\prime}|+\lambda+1)} {(\alpha^{\prime})! \,\Gamma(\ell+\lambda+1)}}\, w^{\alpha^{\prime}}, \hspace{2ex} 
 \mbox{where } \hspace{2ex} \alpha^{\prime} \in \mathbb{Z}^{\ell}_+.
\end{equation*}
Namely, we have
\begin{equation*}
 \mathcal{A}^2_{\lambda}(\mathbb{B}^{\ell}) = \bigoplus_{\rho \in \mathbb{Z}^m_+} \mathscr{H}_{\rho},
\end{equation*}
where, for each $\rho = (\rho_1,...,\rho_m) \in \mathbb{Z}^m_+$, the finite dimensional space $\mathscr{H}_{\rho}$ is defined as
\begin{equation}\label{defn_mathcal_H}
 \mathscr{H}_{\rho} = \textup{span}\left\{ e_{\alpha^{\prime}}^{\lambda} \, : \, \alpha^{\prime}=(\alpha_{(1)}, \cdots, \alpha_{(m)}) 
\in \mathbb{Z}_+^{\ell} \ \ \text{and} \ \ |\alpha_{(j)}|=\rho_j, \ \forall \: j =1,...,m  \right\}.
\end{equation}
For each $p \in \mathbb{Z}_+$, we also introduce the Bergman space $\mathcal{A}^2_{\lambda+p+\ell}(\mathbb{B}^{n-\ell})$ with basis
\begin{equation*}
 e^{\lambda+p+\ell}_{\alpha^{\prime\prime}}(\zeta)
 =\sqrt{\frac{\Gamma(n+|\alpha^{\prime\prime}|+\lambda+p+1)} {(\alpha^{\prime\prime})! \,\Gamma(n+\lambda+p+1)}}\, \zeta^{\alpha^{\prime\prime}},  \hspace{3ex} 
 \alpha^{\prime\prime} \in \mathbb{Z}^{n-\ell}_+. 
\end{equation*}
Given a tuple $\rho = (\rho_1,...,\rho_m) \in \mathbb{Z}^m_+$, we introduce then the linear mapping
\begin{equation*}
 u_{\rho} \, : \ H_{\rho} \ \ \longrightarrow \ \ \mathscr{H}_{\rho} \otimes \mathcal{A}^2_{\lambda+|\rho|+\ell}(\mathbb{B}^{n-\ell}),
\end{equation*}
defined on the basis elements $e^{\lambda}_{\alpha}$, $\alpha = (\alpha^{\prime}, \alpha^{\prime\prime}) \in \mathbb{Z}^{\ell}_+ \times \mathbb{Z}^{n-\ell}_+$, of $H_{\rho}$ as follows
\begin{equation} \label{eq:u_rho}
 u_{\rho} \, : \ e^{\lambda}_{\alpha} \ \ \longmapsto \ \ e^{\lambda}_{\alpha^{\prime}} \otimes e^{\lambda+|\rho|+\ell}_{\alpha^{\prime\prime}}.
\end{equation}
By definition this mapping is an isometric isomorphism. 
\vspace{1ex}\par 
Now introduce the Hilbert space 
\begin{equation} \label{eq:caligraphic_H}
 \mathcal{H} = \bigoplus_{\rho \in \mathbb{Z}^m_+} \mathcal{H}_{\rho},\quad \text{with} \quad \mathcal{H}_{\rho} = \mathscr{H}_{\rho} \otimes \mathcal{A}^2_{\lambda+|\rho|+\ell}(\mathbb{B}^{n-\ell})
\end{equation}
and the unitary operator
\begin{equation} \label{eq:U}
 U = \bigoplus_{\rho \in \mathbb{Z}^m_+} u_{\rho} \, : \ \mathcal{A}_{\lambda}^2(\mathbb{B}^n)=\bigoplus_{\rho \in \mathbb{Z}_+^m} H_{\rho} \ \longrightarrow \ \mathcal{H} = \bigoplus_{\rho \in \mathbb{Z}^m_+} \mathscr{H}_{\rho} \otimes \mathcal{A}^2_{\lambda+|\rho|+\ell}(\mathbb{B}^{n-\ell}),
\end{equation}
acting componentwise according to the direct sum decomposition. Then the above discussion leads to the following proposition.

\begin{prop}\label{Proposition_decomposition_Bergman_space}
 The unitary operator $U$ gives an isometric isomorphism between the spaces
\begin{equation*}
 \mathcal{A}_{\lambda}^2(\mathbb{B}^n)=\bigoplus_{\rho \in \mathbb{Z}_+^m} H_{\rho}
\quad {\rm and} \quad \mathcal{H} = \bigoplus_{\rho \in \mathbb{Z}^m_+} \mathscr{H}_{\rho} \otimes \mathcal{A}^2_{\lambda+|\rho|+\ell}(\mathbb{B}^{n-\ell}).
\end{equation*}
\end{prop}
\section{Toeplitz operators with invariant subspaces} \label{se:inv-subspaces}
\setcounter{equation}{0}
\label{Section_3}
Given a function $f \in L^{\infty}(\mathbb{B}^n)$, we write ${\bf T}_f^{\lambda}$ for the Toeplitz operator with symbol $f$ acting on $\mathcal{A}_{\lambda}^2(\mathbb{B}^n)$. 
More precisely, let 
\begin{equation*}
{\bf P}_{\lambda}: L^2(\mathbb{B}^n, dv_{\lambda}) \longrightarrow \mathcal{A}^2_{\lambda}(\mathbb{B}^n)
\end{equation*}
be the orthogonal projection (Bergman projection). Then ${\bf T}_f^{\lambda}h$ is defined by 
\begin{equation*}
({\bf T}_f^{\lambda}h)(z):= {\bf P}_{\lambda}(fh)(z)= \int_{\mathbb{B}^n} \frac{f(w)h(w)}{(1- z \cdot  \overline{w} )^{\lambda+n+1}} dv_{\lambda}(w), \hspace{4ex}  h \in \mathcal{A}^2_{\lambda}(\mathbb{B}^n). 
\end{equation*}

Throughout the paper we simultaneously use different weighted Bergman spaces on the unit balls of various dimensions, as well as the corresponding Toeplitz operators. To distinguish them, we will always write in {\bf bold} the Toeplitz operators ${\bf T}_f^{\lambda}$ that act on the Bergman space on the ball of the maximal dimension $n$, while we will use the standard font for Toeplitz operators $T_g^{\mu}$ that act on the Bergman space on the ball of a smaller dimension (of dimension $\ell$ or $n-\ell$, in the majority of cases).

\medskip
We start with an auxiliary lemma.

\begin{lem} \label{le:different}
 Let $w = (w^\prime,w^{\prime\prime}) \in \mathbb{B}^k$, and let a function $g \in L^{\infty}(\mathbb{B}^k)$ be invariant under the following action of $\mathbb{T}$ on $\mathbb{B}^k$:
\begin{equation*}
  t \in  \mathbb{T} \, : \ (w^\prime,w^{\prime\prime}) \ \longmapsto \ (tw^\prime,w^{\prime\prime}).
\end{equation*}
Then for multi-indices $\alpha = (\alpha^\prime,\alpha^{\prime\prime})$ and 
$\beta = (\beta^\prime,\beta^{\prime\prime})$ with $|\alpha^\prime| \neq |\beta^\prime|$ we have that 
\begin{equation*}
 \langle g w^{\alpha}, w^{\beta}\rangle_{\mu} = 0.
\end{equation*}
\end{lem}

\begin{proof}
Let $e^{i\theta} = t \in \mathbb{T}$, we calculate
\begin{eqnarray*}
\langle g w^{\alpha}, w^{\beta}\rangle_{\mu} &:=&
 c^{(k)}_{\mu} \int_{\mathbb{B}^k} g(w^\prime,w^{\prime\prime}) w^{\alpha}\overline{w}^{\beta} (1-|w|^2)^{\mu} dv(w) \\
 &=& c^{(k)}_{\mu} \int_{\mathbb{B}^k} g(e^{i\theta}w^\prime,w^{\prime\prime}) w^{\alpha}\overline{w}^{\beta} (1-|w|^2)^{\mu} dv(w) \\
 &=& c^{(k)}_{\mu} \int_{\mathbb{B}^k} g(\zeta^\prime,\zeta^{\prime\prime}) \zeta^{\alpha}\overline{\zeta}^{\beta} e^{i(|\beta^\prime|-|\alpha^\prime|)\theta}(1-|\zeta|^2)^{\mu} dv(\zeta) \\
 &=& e^{i(|\beta^\prime|-|\alpha^\prime|)\theta} \langle g w^{\alpha}, w^{\beta}\rangle_{\mu},
\end{eqnarray*}
here we made a change of variables: $\zeta^{\prime} = e^{i\theta}w^\prime$ and $\zeta^{\prime\prime} = w^{\prime\prime}$. As this holds for all $\theta \in [0,2\pi)$, we get $\langle g w^{\alpha}, w^{\beta}\rangle_{\mu} = 0$.
\end{proof}
{We now present a class of Toeplitz operators on $\mathcal{A}_{\lambda}^2(\mathbb{B}^n)$ which leaves all spaces $H_{\rho}$ in the  decomposition (\ref{Orthogonal_decomposition_of_the_Bergman_space}) invariant. For this we use the separation of coordinates (\ref{GL_separation_of_variables}) and \eqref{definition_z_prime}.} 
\begin{prop} \label{prop:invariant}
 Let $f \in L^{\infty}(\mathbb{B}^n)$ be invariant under the action (\ref{eq:action}) of the group $\mathbb{T}^m \times I \cong 
 \mathbb{T}^m$ on $\mathbb{B}^n$:
\begin{equation} \label{eq:action}
 (t_1,..., t_m) \in \mathbb{T}^m \, : \ (z_{(1)}, \ldots, z_{(m)},\,z^{\prime \prime}) \ \ \longmapsto \ \ (t_1z_{(1)}, \ldots, t_mz_{(m)},\,z^{\prime \prime}). 
\end{equation}
Then the Toeplitz operator ${\bf T}_f^{\lambda}$, acting on $\mathcal{A}^2_{\lambda}(\mathbb{B}^n)$, leaves all the spaces $H_{\rho}$ in the orthogonal decomposition {\rm (\ref{Orthogonal_decomposition_of_the_Bergman_space})} invariant. 
\end{prop}

\begin{proof}
Indeed, Lemma \ref{le:different} implies that $\langle {\bf T}_f^{\lambda} z^{\alpha},
z^{\beta} \rangle_{\lambda} = 0$ for all multi-indices
\begin{equation*}
 \alpha = (\alpha_{(1)}, \ldots, \alpha_{(m)},\,\alpha^{\prime \prime}) \quad \text{and} \quad \beta = (\beta_{(1)}, \ldots, \beta_{(m)},\,\beta^{\prime \prime}),
\end{equation*}
for which there exists $j \in \{1,2,...,m\}$ such that $|\alpha_{(j)}| \neq |\beta_{(j)}|$.

That is, for each $\rho =(\rho_1, \cdots, \rho_m)\in \mathbb{Z}_+^m$, the image ${\bf T}_f^{\lambda}(H_{\rho})$ belongs to $H_{\rho}$.
\end{proof}

We introduce some notation: Given a function $a \in L^{\infty}(\mathbb{B}^{\ell})$ we denote by $f_a \in L^{\infty}(\mathbb{B}^n)$ the function
\begin{equation*}
 f_a(z) = f_a(z^{\prime},z^{\prime\prime}) = a\left(\frac{z^{\prime}}{\sqrt{1-|z^{\prime\prime}|^2}}\right).
\end{equation*}
Similarly, given $c \in L^{\infty}(\mathbb{B}^{n-\ell})$, we define the function $f_c \in L^{\infty}(\mathbb{B}^n)$ by $f_c(z) = f_c(z^{\prime},z^{\prime\prime}) = c(z^{\prime\prime})$.
In what follows we restrict Proposition \ref{prop:invariant} to functions of the form
\begin{equation} \label{eq:f=ac}
 f(z) = f_a(z)f_c(z) = f_{ac}(z)  = a\left(\frac{z^{\prime}}{\sqrt{1-|z^{\prime\prime}|^2}}\right) c(z^{\prime\prime}), 
\end{equation}
where $a \in L^{\infty}(\mathbb{B}^{\ell})$ and $c \in L^{\infty}(\mathbb{B}^{n-\ell})$. Note that $f_{ac}$ is invariant under the action (\ref{eq:action}) if and only if the function $a$ is invariant 
under the following action of the group $\mathbb{T}^m$:
\begin{equation} \label{eq:a-invar}
 (t_1,..., t_m) \in \mathbb{T}^m \, : \ (z_{(1)}, \ldots, z_{(m)}) \ \ \longmapsto \ \ (t_1z_{(1)}, \ldots, t_mz_{(m)}).
\end{equation}

\begin{cor}\label{Corollary_invariant_spaces_orthogonal_decomposition} 
The Toeplitz operator with symbol {\rm (\ref{eq:f=ac})}, with $a$ being invariant under the action {\rm (\ref{eq:a-invar})}, leaves all the spaces $H_{\rho}$ in the orthogonal decomposition {\rm (\ref{Orthogonal_decomposition_of_the_Bergman_space})} invariant.
\end{cor}
We now characterize the action of such Toeplitz operators related to each subspace $H_{\rho}$.
\begin{prop}\label{propositoin_decomposition_TO_form_3_3}
 Let $f_{ac}$ be of the form {\rm (\ref{eq:f=ac})}, with the function $a$ being invariant under the action {\rm (\ref{eq:a-invar})}. Let $T_a^{\lambda}$ be the Toeplitz operator with symbol $a$ acting on $\mathcal{A}^2_{\lambda}(\mathbb{B}^{\ell})$, and let $T_c^{\lambda+|\rho|+\ell}$ be the Toeplitz operator with symbol $c$ acting on $\mathcal{A}^2_{\lambda+|\rho|+\ell}(\mathbb{B}^{n-\ell})$. Then
\begin{equation*}
 u_{\rho} {\bf T}^{\lambda}_{f_{ac}} u_{\rho}^{-1} = T_a^{\lambda}|_{\mathscr{H}_{\rho}} \, \otimes \, T_c^{\lambda+|\rho|+\ell},
\end{equation*}
where the operator $u_{\rho} : H_{\rho}  \longrightarrow \mathscr{H}_{\rho} \otimes \mathcal{A}^2_{\lambda+|\rho|+\ell}(\mathbb{B}^{n-\ell})$ is defined by {\rm (\ref{eq:u_rho})}.
\end{prop}

\begin{proof}
We first mention that by Lemma \ref{le:different} each subspace $\mathscr{H}_{\rho}$ is invariant under the operator~$T_a^{\lambda}$.
Take any two basis elements $e^{\lambda}_{\alpha^{\prime}} \otimes e^{\lambda+|\rho|+\ell}_{\alpha^{\prime\prime}}$ and $e^{\lambda}_{\beta^{\prime}} \otimes e^{\lambda+|\rho|+\ell}_{\beta^{\prime\prime}}$ of 
$\mathscr{H}_{\rho} \otimes \mathcal{A}^2_{\lambda+|\rho|+\ell}(\mathbb{B}^{n-\ell})$, i.e.~$|\alpha_{(j)}| = |\beta_{(j)}| = \rho_j$ for all $j =1,...,m$ and $|\alpha^{\prime}| = |\beta^{\prime}| = |\rho|$. We also set $\alpha = (\alpha^{\prime},\alpha^{\prime\prime})$ and $\beta = (\beta^{\prime},\beta^{\prime\prime})$ and calculate
\begin{eqnarray} \label{eq:part_1} \nonumber
 && \left\langle u_{\rho} {\bf T}^{\lambda}_{f_{ac}} u_{\rho}^{-1} (e^{\lambda}_{\alpha^{\prime}} \otimes e^{\lambda+|\rho|+\ell}_{\alpha^{\prime\prime}}), (e^{\lambda}_{\beta^{\prime}} \otimes e^{\lambda+|\rho|+\ell}_{\beta^{\prime\prime}}) \right\rangle_{\lambda}  \\ \nonumber
 && = \left\langle{\bf T}^{\lambda}_{f_{ac}} e^{\lambda}_{\alpha},  e^{\lambda}_{\beta}\right\rangle_{\lambda} 
 = \left\langle f_{ac}\, e^{\lambda}_{\alpha},  e^{\lambda}_{\beta}\right\rangle_{\lambda} \\
&& = \sqrt{\frac{\Gamma(n+|\alpha|+\lambda+1)\Gamma(n+|\beta|+\lambda+1)} {\alpha!\, \beta!\, \Gamma(n+\lambda+1)^2}} \, \left\langle f_{ac}\, z^{\alpha},  z^{\beta}\right\rangle_{\lambda}.
\end{eqnarray}
Then (we will do a change of variables: $z^{\prime} = w \sqrt{1-|\zeta|^2}$, \ $z^{\prime\prime} = \zeta$):
\begin{eqnarray} \label{eq:part_2} \nonumber
 && \left\langle f_{ac}\, z^{\alpha},  z^{\beta}\right\rangle_{\lambda} = 
c^{(n)}_{\lambda} \int_{\mathbb{B}^n} f_{ac}\,z^{\alpha}\overline{z}^{\beta} (1-|z|^2)^{\lambda}\,dv(z) \\ \nonumber
&& = \ c^{(n)}_{\lambda} \int_{\mathbb{B}^{\ell}\times \mathbb{B}^{n-\ell}} a(w)c(\zeta) \, w^{\alpha^{\prime}}\overline{w}^{\beta^{\prime}} \,  \zeta^{\alpha^{\prime\prime}}\overline{\zeta}^{\beta^{\prime\prime}}\\ \nonumber
&& \times \ (1-|w|^2)^{\lambda} (1-|\zeta|^2)^{\lambda+|\rho|+\ell}\,dv(w)dv(\zeta) \\ \nonumber
&& = \ \frac{c^{(n)}_{\lambda}}{c^{(\ell)}_{\lambda}c^{(n-\ell)}_{\lambda+|\rho|+\ell}}\, \int_{\mathbb{B}^{\ell}} a(w) \, w^{\alpha^{\prime}}\overline{w}^{\beta^{\prime}}\,dv_{\lambda}(w) \int_{\mathbb{B}^{n-\ell}} c(\zeta) \,  \zeta^{\alpha^{\prime\prime}}\overline{\zeta}^{\beta^{\prime\prime}}\,dv_{\lambda+|\rho|+\ell}(\zeta) \\ \nonumber
&& = \ \frac{\Gamma(n+\lambda+1)\Gamma(\ell+|\rho|+\lambda+1)}{\Gamma(\ell+\lambda+1)\Gamma(n+|\rho|+\lambda+1)}\, \left\langle a\, w^{\alpha^{\prime}}, w^{\beta^{\prime}} \right\rangle_{\lambda} 
\cdot \left\langle c\, \zeta^{\alpha^{\prime\prime}}, \zeta^{\beta^{\prime\prime}} \right\rangle_{\lambda+ |\rho|+ \ell } \\ \nonumber
&& = \frac{\Gamma(n+\lambda+1)\Gamma(\ell+|\rho|+\lambda+1)}{\Gamma(\ell+\lambda+1)\Gamma(n+|\rho|+\lambda+1)}\\ \nonumber
&& \times \ \sqrt{\frac{(\alpha^{\prime})!\, (\beta^{\prime})!\, [\Gamma(\ell+\lambda+1)]^2}{\Gamma(\ell+|\alpha^{\prime}|+\lambda+1) \Gamma(\ell+|\beta^{\prime}|+\lambda+1)}}\,\left\langle a\, e^{\lambda}_{\alpha^{\prime}}, e^{\lambda}_{\beta^{\prime}} \right\rangle_{\lambda} \\
&& \times \ \sqrt{\frac{(\alpha^{\prime\prime})!\, (\beta^{\prime\prime})!\, [\Gamma(n+\lambda+|\alpha^{\prime}|+1)]^2}{\Gamma(n+|\alpha|+\lambda+1) \Gamma(n+|\beta|+\lambda+1)}}\,\left\langle c\, e^{\lambda+|\rho|+\ell}_{\alpha^{\prime\prime}}, e^{\lambda+|\rho|+\ell}_{\beta^{\prime\prime}} \right\rangle_{\lambda+|\rho|+ \ell}.
\end{eqnarray} 
Comparing (\ref{eq:part_1}) and (\ref{eq:part_2}), we obtain
\begin{multline*}
 \left\langle u_{\rho} {\bf T}^{\lambda}_{f_{ac}} u_{\rho}^{-1} (e^{\lambda}_{\alpha^{\prime}} \otimes e^{\lambda+|\rho|+\ell}_{\alpha^{\prime\prime}}), (e^{\lambda}_{\beta^{\prime}} \otimes e^{\lambda+|\rho|+\ell}_{\beta^{\prime\prime}}) \right\rangle_{\lambda}  =\\=
  \left\langle T^{\lambda}_a\, e^{\lambda}_{\alpha^{\prime}}, e^{\lambda}_{\beta^{\prime}} \right\rangle_{\lambda} \cdot \left\langle T^{\lambda+|\rho|+\ell}_c\, e^{\lambda+|\rho|+\ell}_{\alpha^{\prime\prime}}, e^{\lambda+|\rho|+\ell}_{\beta^{\prime\prime}} \right\rangle_{\lambda+|\rho|+ \ell}
\end{multline*}
and the result follows.
\end{proof}

In the next corollaries we characterize the action of a Toeplitz operator ${\bf T}^{\lambda}_{f_{ac}}$ related to the direct sum decomposition (\ref{eq:caligraphic_H}), as well as some of its properties.

\begin{cor} \label{co:f_ac}
Under the assumptions of Proposition \ref{propositoin_decomposition_TO_form_3_3} we have
\begin{equation*}
 U {\bf T}^{\lambda}_{f_{ac}} U^{-1} = \bigoplus_{\rho \in \mathbb{Z}^m_+}\,  T_a^{\lambda}|_{\mathscr{H}_{\rho}} \, \otimes \, T_c^{\lambda+|\rho|+\ell},
\end{equation*}
where the operator $U$ is given by {\rm (\ref{eq:U})}. In particular, 
\begin{eqnarray*}
 U {\bf T}^{\lambda}_{f_{a}} U^{-1} &=& \bigoplus_{\rho \in \mathbb{Z}^m_+}\,  T_a^{\lambda}|_{\mathscr{H}_{\rho}} \, \otimes \, I, \\
 U {\bf T}^{\lambda}_{f_{c}} U^{-1} &=& \bigoplus_{\rho \in \mathbb{Z}^m_+}\,  I \, \otimes \, T_c^{\lambda+|\rho|+\ell}, \\
 {\bf T}^{\lambda}_{f_{ac}} &=& {\bf T}^{\lambda}_{f_{a}} {\bf T}^{\lambda}_{f_{c}} = 
 {\bf T}^{\lambda}_{f_{c}} {\bf T}^{\lambda}_{f_{a}}.
\end{eqnarray*}
\end{cor}

Let $H$ be a Hilbert space with orthogonal decomposition 
\begin{equation*}
H=\bigoplus_{j=1}^{\infty} H_j
\end{equation*}
and assume that $A \in \mathcal{L}(H)$ leaves $H_j$ invariant for all $j \in \mathbb{N}$. We write $A_j$ for the restriction of $A$ to $H_j$. Then 
it can be easily seen that 
\begin{equation*}
\|A\|= \sup_{j \in \mathbb{N}} \|A_j\|. 
\end{equation*}
\begin{cor}\label{cor_Theorem_1_2 }
With the above assumptions on symbols, we have
\begin{equation*}
\|{\bf T}^{\lambda}_{f_{ac}}\|=\sup_{\rho \in \mathbb{Z}^m_+} \, \|T_a^{\lambda}|_{\mathscr{H}_{\rho}}\|\cdot \|T_c^{\lambda+|\rho|+\ell}\|.
\end{equation*}
\end{cor}

\begin{rem} \label{re:abrev} {\rm
Let ${\bf T}^{\lambda}$ be a bounded operator on $\mathcal{A}^2_{\lambda}(\mathbb{B}^n)$ which leaves all spaces $H_{\rho}$, $\rho \in \mathbb{Z}^m_+$, invariant and set $\mathcal{H}_{\rho} = \mathscr{H}_{\rho} \otimes \mathcal{A}^2_{\lambda+|\rho|+\ell}(\mathbb{B}^{n-\ell})$. Then there exists a sequence  $\{T_{\rho}\}_{\rho \in \mathbb{Z}^m_+}\subset \mathcal{L}(\mathcal{H}_{\rho})$ such that
\begin{equation} \label{eq:o-plus}
 U {\bf T}^{\lambda}U^* = \bigoplus_{\rho \in \mathbb{Z}^m_+} T_{\rho}
\end{equation}
via the unitary operator $U$ given by (\ref{eq:U}). In what follows we will abbreviate (\ref{eq:o-plus}) by
\begin{equation*}
 {\bf T}^{\lambda} \asymp \bigoplus_{\rho \in \mathbb{Z}^m_+} T_{\rho},
\end{equation*}
identifying the operator ${\bf T}^{\lambda}$ with its direct sum representation.  }
\end{rem}
We also recall the definition of the Berezin transform 
\begin{equation*}
\mathcal{B}_{\mu}\ : \ \mathcal{L}(\mathcal{A}^2_{\mu}(\mathbb{B}^{d})) \rightarrow L^{\infty}(\mathbb{B}^{d})\ : \ \mathcal{B}_{\mu}[A](z)\, :=  \, \big{\langle} Ak_z^{\mu},k_z^{\mu} \big{\rangle}_{\mu}.
\end{equation*}
Here $k_z^{\mu}$ denotes the normalized reproducing kernel function of $\mathcal{A}^2_{\mu}(\mathbb{B}^{d})$ defined by
\begin{equation*}
 k_z^{\mu}(w)=\frac{(1-|z|^2)^{\frac{d+\mu+1}{2}}}{(1-w \cdot \overline{z})^{d+\mu+1}}, \hspace{5ex} z,w \in \mathbb{B}^{d}. 
\end{equation*}
For the special case of a Toeplitz operator $T^{\mu}_g$, we have
\begin{equation*}
 \mathcal{B}_{\mu}[T^{\mu}_g](z)  =  \mathcal{B}_{\mu}[g](z) = \big{\langle} g\,k_z^{\mu},k_z^{\mu} \big{\rangle}_{\mu}.
\end{equation*}
In this paper we aim to analyze operator algebras that are generated by certain families of Toeplitz operators (also see the results in \cite{AA,BV1,BV1-II,Le_2010, MiSuWi, Suarez_2007,V1,XZ}). 
For general  symbols $a \in L^{\infty}(\mathbb{B}^{\ell})$ and $c \in L^{\infty}(\mathbb{B}^{n-\ell})$ not much can be said about the structure of the algebra generated by all Toeplitz operators ${\bf T}^{\lambda}_{f_{ac}}$. Thus for the rest of the paper our strategy will be as follows. We select subclasses $\mathcal{S}_{\ell} \subset L^{\infty}(\mathbb{B}^{\ell})$ of functions invariant under the action (\ref{eq:a-invar}) of the group $\mathbb{T}^m$ and $\mathcal{S}_{n-\ell} \subset L^{\infty}(\mathbb{B}^{n-\ell})$. Then we consider the closed unital algebra 
$\boldsymbol{\mathcal{T}}_{\lambda}(\mathcal{S}_{\ell}, \mathcal{S}_{n-\ell})\subset \mathcal{L}(\mathcal{A}_{\lambda}^2(\mathbb{B}^n))$ 
generated by all Toeplitz operators ${\bf T}^{\lambda}_{f_{ac}}$ with $a \in \mathcal{S}_{\ell}$ and $c \in \mathcal{S}_{n-\ell}$.

By Corollary \ref{co:f_ac} the algebra $\boldsymbol{\mathcal{T}}_{\lambda}(\mathcal{S}_{\ell}, \mathcal{S}_{n-\ell})$ is generated by its subalgebras 
$\boldsymbol{\mathcal{T}}_{\lambda}(\mathcal{S}_{\ell}, 1)$ and $\boldsymbol{\mathcal{T}}_{\lambda}(1, \mathcal{S}_{n-\ell})$. Moreover, elements ${\bf T}^{\prime} \in \boldsymbol{\mathcal{T}}_{\lambda}(\mathcal{S}_{\ell}, 1)$ commute with elements ${\bf T}^{\prime\prime} \in \boldsymbol{\mathcal{T}}_{\lambda}(1, \mathcal{S}_{n-\ell})$. 
\vspace{0.5ex} \par 
In the next sections we will select  $\mathcal{S}_{\ell}$  and $\mathcal{S}_{n-\ell}$ for which $\boldsymbol{\mathcal{T}}_{\lambda}(\mathcal{S}_{\ell}, 1)$ and $\boldsymbol{\mathcal{T}}_{\lambda}(1, \mathcal{S}_{n-\ell})$ admit {\it ``reasonable descriptions''}. Then we will characterize the algebra $\boldsymbol{\mathcal{T}}_{\lambda}(\mathcal{S}_{\ell}, \mathcal{S}_{n-\ell})$.
\section{Operators with BUC symbols}
\setcounter{equation}{0}
\label{Section_4}
Let us denote by ${\rm BUC}(\mathbb{B}^{n-\ell})$ the space ($C^*$-algebra) of all bounded complex-valued functions on $\mathbb{B}^{n-\ell}$ that are uniformly continuous with respect to the Bergman metric on $\mathbb{B}^{n-\ell}$ (cf. \cite{BC1}).

Introduce the $C^*$-algebra $\boldsymbol{\mathcal{T}}_{\lambda}(1,{\rm BUC}(\mathbb{B}^{n-\ell}))$, which is generated by all Toeplitz operators $\mathbf{T}^{\lambda}_{f_c}$ acting on $\mathcal{A}_{\lambda}^2(\mathbb{B}^n)$ with $c \in {\rm BUC}(\mathbb{B}^{n-\ell})$ and denote by $\mathcal{T}_{\mu}({\rm BUC}(\mathbb{B}^{n-\ell}))$ the $C^*$-algebra generated by all Toeplitz operators $T^{\mu}_c$ acting on $\mathcal{A}_{\mu}^2(\mathbb{B}^{n-\ell})$ with $c \in {\rm BUC}(\mathbb{B}^{n-\ell})$.
\begin{rem}{\rm 
We note that the algebra $\mathcal{T}_{\mu}({\rm BUC}(\mathbb{B}^{n-\ell}))$ (by \cite[Theorem 7.3]{Suarez_2007}) coincides with the $C^*$-algebra $\mathcal{T}_{\mu}(L^{\infty}(\mathbb{B}^{n-\ell}))$ 
generated by all Toeplitz operators $T^{\mu}_c$  acting on $\mathcal{A}_{\mu}^2(\mathbb{B}^{n-\ell})$ with $c \in L^{\infty}(\mathbb{B}^{n-\ell})$.}
\end{rem}
Each operator $\mathbf{T}^{\lambda} \in \boldsymbol{\mathcal{T}}_{\lambda}(1,{\rm BUC}(\mathbb{B}^{n-\ell}))$ admits the representation
\begin{equation} \label{eq:T_dir-sum}
 \mathbf{T}^{\lambda} \asymp \bigoplus_{\rho \in \mathbb{Z}_+^m} I \otimes T^{\lambda+|\rho|+\ell},
\end{equation}
where  $T^{\lambda+|\rho|+\ell} \in \mathcal{T}_{\lambda+|\rho|+\ell}({\rm BUC}(\mathbb{B}^{n-\ell}))$, for each $\rho \in \mathbb{Z}_+^m$. 
\begin{lem} \label{le:sup-norm}
For each operator $\mathbf{T}^{\lambda} \in \boldsymbol{\mathcal{T}}_{\lambda}(1,{\rm BUC}(\mathbb{B}^{n-\ell}))$ we have that
\begin{equation*}
 \|\mathbf{T}^{\lambda}\| = \sup_{\rho \in \mathbb{Z}_+^m} \|T^{\lambda+|\rho|+\ell}\|, \hspace{2ex} \mbox{and} \hspace{2ex} 
\|{\bf T}_{f_c}^{\lambda}\|=\sup_{z \in \mathbb{B}^{n-\ell}} |c(z)|=\| c\|_{\infty}.
\end{equation*}
\end{lem}

\begin{proof}
The first equality follows from Corollary \ref{cor_Theorem_1_2 }. The second equality follows from \cite[Proposition 4.4]{BC1} and
\begin{equation*}
 \|c\|_{\infty} =\| \lim_{j \rightarrow \infty} \mathcal{B}_{\lambda+j+\ell}\big{[}T_c^{\lambda+j+\ell}\big{]}(z)\|_{\infty}
\leq \sup_{j\in \mathbb{Z}_+} \big{\|} T_c^{\lambda+j+\ell} \big{\|} = \|{\bf T}_{f_c}^{\lambda}\| \leq
\|c\|_{\infty}. \qedhere
\end{equation*} 
\end{proof}
\begin{thm} \label{th:rep}
Each operator $\mathbf{T}^{\lambda} \in \boldsymbol{\mathcal{T}}_{\lambda}(1,{\rm BUC}(\mathbb{B}^{n-\ell}))$, in the direct sum decomposition {\rm (\ref{eq:T_dir-sum})}, 
admits the unique representation
\begin{equation} \label{eq:unique}
  \mathbf{T}^{\lambda}
\asymp \bigoplus_{\rho \in \mathbb{Z}_+^m} I \otimes (T_c^{\lambda+|\rho|+\ell} + N^{\lambda+|\rho|+\ell}),
\end{equation}
where $c \in {\rm BUC}(\mathbb{B}^{n-\ell})$, each operator $N^{\lambda+|\rho|+\ell}$ belongs to the semi-commutator ideal of the algebra 
$\mathcal{T}_{\lambda+|\rho|+\ell}({\rm BUC}(\mathbb{B}^{n-\ell}))$, and $\|N^{\lambda+|\rho|+\ell}\| \to 0$ as $|\rho| \to \infty$.
\end{thm}
\begin{proof}
We consider first a dense subalgebra $\boldsymbol{\mathcal{D}}_{\lambda}$ of $\boldsymbol{\mathcal{T}}_{\lambda}(1,{\rm BUC}(\mathbb{B}^{n-\ell}))$, which consists of 
finite sums of finite products of initial generators. To prove the statement for this case, it is sufficient to prove it for finite products of $m$ Toeplitz operators. By induction we can 
assume that $m=2$. Thus, given $f_a,\, f_b$ with $a,\, b  \in {\rm BUC}(\mathbb{B}^{n-\ell})$, we consider
\begin{equation*}
 \mathbf{T}_{f_a}^{\lambda}\mathbf{T}_{f_b}^{\lambda} \asymp \bigoplus_{\rho \in \mathbb{Z}_+^m} I \otimes T_a^{\lambda+|\rho|+\ell}T_b^{\lambda+|\rho|+\ell}.
\end{equation*}
For each $\rho$, we write $T_a^{\lambda+|\rho|+\ell}T_b^{\lambda+|\rho|+\ell} = 
T_{ab}^{\lambda+|\rho|+\ell} + N^{\lambda+|\rho|+\ell}$, where the operator 
$$N^{\lambda+|\rho|+\ell} = T_a^{\lambda+|\rho|+\ell}T_b^{\lambda+|\rho|+\ell} - T_{ab}^{\lambda+|\rho|+\ell}$$ 
belongs to the semi-commutator ideal of the algebra $\mathcal{T}_{\lambda+|\rho|+\ell}({\rm BUC}(\mathbb{B}^{n-\ell}))$. Finally, \cite[Theorem 3.8]{BaHaVa} implies that 
$\|N^{\lambda+|\rho|+\ell}\| \to 0$ as $|\rho| \to \infty$.
\vspace{1ex}\par 
Now, given $\mathbf{T}^{\lambda} \in \boldsymbol{\mathcal{T}}_{\lambda}(1,{\rm BUC}(\mathbb{B}^{n-\ell}))$, there exists a fundamental sequence $\{\mathbf{T}^{\lambda, k}\}_{k \in \mathbb{N}}$ of elements 
from $\boldsymbol{\mathcal{D}}_{\lambda}$ that converges in norm to $\mathbf{T}^{\lambda}$. Each $\mathbf{T}^{\lambda,k}$ has the form
\begin{equation*}
\mathbf{T}^{\lambda, k} \asymp \bigoplus_{\rho \in \mathbb{Z}_+^m} I \otimes T^{\lambda+|\rho|+\ell,k}, \quad \textup{with} \quad 
T^{\lambda+|\rho|+\ell,k} = T_{a_k}^{\lambda+|\rho|+\ell} + N_k^{\lambda+|\rho|+\ell},
\end{equation*}
where $a_k \in {\rm BUC}(\mathbb{B}^{n-\ell})$ for each $k \in \mathbb{N}$ and $N_k^{\lambda+|\rho|+\ell}$ belongs to the semi-commutator ideal and with 
$N_k^{\lambda+|\rho|+\ell} \to 0$ as $|\rho| \to \infty$, for fixed $k$ . By Lemma \ref{le:sup-norm}, each sequence 
$\{T^{\lambda+|\rho|+\ell,k}\}_{k \in \mathbb{N}}$ is also fundamental in $\mathcal{L}(\mathcal{A}^2_{\lambda+|\rho|+\ell}(\mathbb{B}^{n-\ell}))$.

In particular, for any $\varepsilon > 0$ there exists $N_0 \in \mathbb{N}$ such that for all $k,\, m > N_0$ we have the following estimate (uniformly in $\rho$): 
\begin{equation*}
 \|T^{\lambda+|\rho|+\ell,k} - T^{\lambda+|\rho|+\ell,m}\| \leq \|\mathbf{T}^{\lambda,k} - \mathbf{T}^{\lambda,m}\| < \textstyle{\frac{\varepsilon}{2}}.
\end{equation*}
Observe now that $T^{\lambda+|\rho|+\ell,k} - T^{\lambda+|\rho|+\ell,m} = T_{(a_k-a_m)}^{\lambda+|\rho|+\ell} + (N_k^{\lambda+|\rho|+\ell} - N_m^{\lambda+|\rho|+\ell})$. For any fixed $k,\, m > N_0$ we 
pass to the limit as $|\rho| \to \infty$. Then, taking into account Lemma \ref{le:sup-norm} together with the observation that both $\|N_k^{\lambda+|\rho|+\ell}\|$ and $\|N_m^{\lambda+|\rho|+\ell}\|$ tend to 
$0$ as $|\rho| \to \infty$, we have
\begin{equation*}
 \lim_{|\rho| \to \infty} \|T^{\lambda+|\rho|+\ell,k} - T^{\lambda+|\rho|+\ell,m}\| =
 \lim_{|\rho| \to \infty} \|T_{(a_k-a_m)}^{\lambda+|\rho|+\ell}\| = \|a_k-a_m\|_{\infty} \leq \textstyle{\frac{\varepsilon}{2}}.
\end{equation*}
Hence the function sequence $\{a_k\}_{k \in \mathbb{N}}$ is fundamental, and thus it converges to some $a \in {\rm BUC}(\mathbb{B}^{n-\ell})$. Then, $\|T_{a_k}^{\lambda+|\rho|+\ell} - T_a^{\lambda+|\rho|+\ell}\| \leq \|a_k-a\|_{\infty}$ implies that, for each fixed $\rho$, the sequence $\{T_{a_k}^{\lambda+|\rho|+\ell}\}_{k \in \mathbb{N}}$ converges in norm to the operator $T_a^{\lambda+|\rho|+\ell}$. Thus, for each fixed $\rho$, the sequence of 
operators $\{N_k^{\lambda+|\rho|+\ell}\}_{k \in \mathbb{N}}$, being the difference of two convergent sequences $\{T^{\lambda+|\rho|+\ell,k}\}_{k \in \mathbb{N}}$ and \\ $\{T_{a_k}^{\lambda+|\rho|+\ell}\}_{k \in \mathbb{N}}$, converges in norm to an operator $N^{\lambda+|\rho|+\ell}$ from the semi-commutator ideal. 

This implies the desired representation 
\begin{equation*}
  \mathbf{T}^{\lambda} \asymp \bigoplus_{\rho \in \mathbb{Z}_+^m} I \otimes T^{\lambda+|\rho|+\ell} = \bigoplus_{\rho \in \mathbb{Z}_+^m} I \otimes (T_a^{\lambda+|\rho|+\ell} + N^{\lambda+|\rho|+\ell}).
\end{equation*}
\par 
It remains to prove that $\|N^{\lambda+|\rho|+\ell}\| \to 0$ as $|\rho| \to \infty$. To do this we use the standard $\frac{\varepsilon}{3}$-trick. Using the representation
$$T^{\lambda+|\rho|+\ell} - T^{\lambda+|\rho|+\ell,k} = T_{(a-a_k)}^{\lambda+|\rho|+\ell} + 
N^{\lambda+|\rho|+\ell} - N_k^{\lambda+|\rho|+\ell},$$ 
we obtain:
\begin{equation*}
 \|N^{\lambda+|\rho|+\ell}\| \leq \|T^{\lambda+|\rho|+\ell} - T^{\lambda+|\rho|+\ell,k}\| + \| T_{(a-a_k)}^{\lambda+|\rho|+\ell}\| + \|N_k^{\lambda+|\rho|+\ell}\|.
\end{equation*}
Now, given any $\varepsilon > 0$, there exists $k \in \mathbb{N}$ such that
\begin{equation*}
 \|T^{\lambda+|\rho|+\ell} - T^{\lambda+|\rho|+\ell,k}\| \leq \|\mathbf{T}^{\lambda} - \mathbf{T}^{\lambda,k}\| < \textstyle{\frac{\varepsilon}{3}} \quad \textup{ and} \quad
 \| T_{(a-a_k)}^{\lambda+|\rho|+\ell}\| \leq \|a - a_k\|_{\infty} < \textstyle{\frac{\varepsilon}{3}},
\end{equation*}
both uniformly in $\rho$. With this fixed $k$ and $\frac{\varepsilon}{3}$, there exists $N_0 \in \mathbb{N}$ such that for all $|\rho| > N_0$ we have that $\|N_k^{\lambda+|\rho|+\ell}\| < \frac{\varepsilon}{3}$. The above implies that for all $|\rho| > N_0$ we have
\begin{equation*}
 \|N^{\lambda+|\rho|+\ell}\| < 3 \, \textstyle{\frac{\varepsilon}{3}} = \varepsilon.
\end{equation*}

To prove the uniqueness of the representation assume that 
\begin{equation*}
 \mathbf{T}^{\lambda} \asymp \bigoplus_{\rho \in \mathbb{Z}_+^m} I \otimes (T_{a_1}^{\lambda+|\rho|+\ell} + N_1^{\lambda+|\rho|+\ell}) = \bigoplus_{\rho \in \mathbb{Z}_+^m} I \otimes (T_{a_2}^{\lambda+|\rho|+\ell} + N_2^{\lambda+|\rho|+\ell}),
\end{equation*}
where $a_j \in {\rm BUC}(\mathbb{B}^{n-\ell})$, each $N_j^{\lambda+|\rho|+\ell}$ belongs to the semi-commutator ideal, and $N_j^{\lambda+|\rho|+\ell} \to 0$ as $|\rho| \to \infty$, for $j=1,2$. Then 
\begin{equation} \label{eq:0}
 T_{a_1-a_2}^{\lambda+|\rho|+\ell} + (N_1^{\lambda+|\rho|+\ell} - N_2^{\lambda+|\rho|+\ell}) = 0, \qquad \textup{ for \ each} \quad \rho \in \mathbb{Z}_+^m,
\end{equation}
and thus
\begin{equation*}
 0 = \lim_{|\rho| \rightarrow \infty} \|T_{a_1-a_2}^{\lambda+|\rho|+\ell} + (N_1^{\lambda+|\rho|+\ell} - N_2^{\lambda+|\rho|+\ell})\| = \lim_{|\rho| \rightarrow \infty} \|T_{a_1-a_2}^{\lambda+|\rho|+\ell} \| = \|a_1 - a_2\|_{\infty}.
\end{equation*}
Thus $a_1 = a_2$, and (\ref{eq:0}) implies that $N_1^{\lambda+|\rho|+\ell} = N_2^{\lambda+|\rho|+\ell}$ for all $|\rho| \in \mathbb{Z}_+$.
\end{proof}

\begin{rem} \label{rem:no-one-dim} {\rm
As it follows from \cite{Le_2010}, for each weight parameter $\mu > -1$, the  commutator ideal of the algebra $\mathcal{T}_{\mu}({\rm BUC}(\mathbb{B}^{n-\ell})) 
= \mathcal{T}_{\mu}(L^{\infty}(\mathbb{B}^{n-\ell}))$ coincides with the whole algebra $\mathcal{T}_{\mu}(L^{\infty}(\mathbb{B}^{n-\ell}))$.
 In particular, this implies that the $C^*$-algebra $\mathcal{T}_{\mu}({\rm BUC}(\mathbb{B}^{n-\ell}))$ does not have non-trivial one-dimensional representations. }
\end{rem}

It is instructive to see that, although $\boldsymbol{\mathcal{T}}_{\lambda}({\rm BUC}(\mathbb{B}^{n})) = \boldsymbol{\mathcal{T}}_{\lambda}(L^{\infty}(\mathbb{B}^{n}))$, 
we have the strict inclusion
\begin{equation}\label{Strict_inclusion_algebras}
\boldsymbol{\mathcal{T}}_{\lambda}(1,{\rm BUC}(\mathbb{B}^{n-\ell})) \subsetneq \boldsymbol{\mathcal{T}}_{\lambda}(1,L^{\infty}(\mathbb{B}^{n-\ell})).
\end{equation}
In fact, this follows from the next result: 
\begin{lem}
 If $a \in L^{\infty}(\mathbb{B}^{n-\ell})$ and $\mathbf{T}^{\lambda}_{f_a} \in \boldsymbol{\mathcal{T}}_{\lambda}(1,{\rm BUC}(\mathbb{B}^{n-\ell}))$, then necessarily the function $a$ belongs to ${\rm BUC}(\mathbb{B}^{n-\ell})$.
\end{lem}
\begin{proof}
By Theorem \ref{th:rep}, we have
\begin{equation*}
 \mathbf{T}^{\lambda}_{f_a} \asymp \bigoplus_{\rho \in \mathbb{Z}_+^m} I \otimes T^{\lambda+|\rho|+\ell}_a = \bigoplus_{\rho \in \mathbb{Z}_+^m} I \otimes (T_c^{\lambda+|\rho|+\ell} + N^{\lambda+|\rho|+\ell}),
\end{equation*}
for some $c \in {\rm BUC}(\mathbb{B}^{n-\ell})$, and $\|N^{\lambda+|\rho|+\ell}\| \to 0$ as $|\rho| \to \infty$.

For simplicity we let $\mu:= \lambda+|\rho|+\ell$. Then
\begin{equation} \label{eq:B(a)=B(c)}
 c = \lim_{\mu \to \infty} \mathcal{B}_{\mu}[c] = \lim_{\mu \to \infty} \mathcal{B}_{\mu}[T_c^{\mu} + N^{\mu}] = \lim_{\mu \to \infty}  \mathcal{B}_{\mu}[T^{\mu}_a] = \lim_{\mu \to \infty} \mathcal{B}_{\mu}[a].
\end{equation}
Let $\varphi_z$ denote the usual automorphism (M\"obius transform) of $\mathbb{B}^{n-\ell}$ hat interchanges $0$ and $z$. Besides the standard Berezin transform
\begin{equation*}
 \mathcal{B}_{\mu}[a](z) = \mathcal{B}_{\mu}[T^{\mu}_a](z) = \int_{\mathbb{B}^{n-\ell}} a \circ \varphi_z(w) \, dv_{\mu}(w),
\end{equation*}
for a Toeplitz operator $T^{\mu}_a$, with $a \in L^{\infty}(\mathbb{B}^{n-\ell})$, we will consider the {\it $(m,\mu)$-Berezin transform} (see, \cite{BHV} for details), defined for 
Toeplitz operators $T^{\mu}_a$ with $a \in L^{\infty}(\mathbb{B}^{n-\ell})$ as follows:
\begin{equation*}
 \mathcal{B}_{m,\mu}[T^{\mu}_a](z) = \mathcal{B}_{m,\mu}[a](z) = \int_{\mathbb{B}^{n-\ell}} a \circ \varphi_z(w) \, dv_{m+\mu}(w).
\end{equation*}
Note that $\mathcal{B}_{\mu}(T^{\mu}_a) = \mathcal{B}_{0,\mu}(T^{\mu}_a)$ and
$\mathcal{B}_{m,\mu}(T^{\mu}_a) = \mathcal{B}_{0,m+\mu}(T^{\mu}_a)$.

Estimate now
\begin{eqnarray*}
 0 &\leq& \|T^{\mu}_a - T^{\mu}_c\|  \leq \|T^{\mu}_a - T^{\mu}_{\mathcal{B}_{m,\mu}(a)}\| + \|T^{\mu}_{\mathcal{B}_{m,\mu}(a)} - T^{\mu}_{\mathcal{B}_{m,\mu}(c)}\| + \|T^{\mu}_{\mathcal{B}_{m,\mu}(c)} - T^{\mu}_c\| \\
 &\leq& \|T^{\mu}_a - T^{\mu}_{\mathcal{B}_{m,\mu}(a)}\| + \|\mathcal{B}_{m,\mu}(a) - \mathcal{B}_{m,\mu}(c) \|_{\infty} + \|\mathcal{B}_{m,\mu}(c) - c\|_{\infty}.
\end{eqnarray*}
Let $m \to \infty$, then $\|T^{\mu}_a - T^{\mu}_{\mathcal{B}_{m,\mu}(a)}\| \to 0$ by \cite[Theorem A.1]{BHV}. Moreover, 
\begin{align*}
\|\mathcal{B}_{m,\mu}(a) - \mathcal{B}_{m,\mu}(c) \|_{\infty} &= \|\mathcal{B}_{0,m+\mu}(a) - \mathcal{B}_{0,m+\mu}(c) \|_{\infty} \to 0, \hspace{3ex} (\textup{by (\ref{eq:B(a)=B(c)})}) \\
\|\mathcal{B}_{m,\mu}(c) - c\|_{\infty} &= \|\mathcal{B}_{0,m+\mu}(c) - c\|_{\infty} \to 0 \hspace{3ex} (\textup{by  \cite[Proposition 4.4]{BC1}}).
\end{align*}
That is, $\|T^{\mu}_a - T^{\mu}_c\| = \|T^{\mu}_{a-c}\| = 0$ and thus $a = c \in {\rm BUC}(\mathbb{B}^{n-\ell})$. 
\end{proof}

Note that for $a \in L^{\infty}(\mathbb{B}^{n-\ell})$ the function $f_a$ belongs to $L^{\infty}(\mathbb{B}^{n})$ and thus $\mathbf{T}^{\lambda}_{f_a}$ can be norm approximated by Toeplitz operators with ${\rm BUC}(\mathbb{B}^{n})$-symbols, but, as the previous lemma shows, it cannot be done with ${\rm BUC}(\mathbb{B}^{n- \ell})$-symbols. This shows the strict inclusion (\ref{Strict_inclusion_algebras}). 

For each $\mu = \lambda + |\rho| + \ell$, denote by $$\mathcal{T}^{(n-\ell)}_{\mu}({\rm BUC}(\mathbb{B}^{n-\ell}))$$ the algebra that consists of all operators $T^{\mu} = T^{\mu}_c + N^{\mu}$ (see (\ref{eq:T_dir-sum}) and (\ref{eq:unique})) coming from the restriction of the algebra $\boldsymbol{\mathcal{T}}_{\lambda}(1,{\rm BUC}(\mathbb{B}^{n-\ell}))$ onto its invariant subspaces.

\begin{prop} \label{prop:coincide}
 The algebra $\mathcal{T}^{(n-\ell)}_{\mu}({\rm BUC}(\mathbb{B}^{n-\ell}))$ coincides with the whole $C^*$-algebra
\[\mathcal{T}_{\mu}({\rm BUC}(\mathbb{B}^{n-\ell})) = \mathcal{T}_{\mu}(L^{\infty}(\mathbb{B}^{n-\ell})).\]
\end{prop}

\begin{proof}
 Given $\mu = \lambda + k + \ell$, observe, first, that the mapping
\begin{eqnarray*}
 \iota_k&:& \boldsymbol{\mathcal{T}}_{\lambda}(1,{\rm BUC}(\mathbb{B}^{n-\ell})) \ \longrightarrow \ \mathcal{L}\big{(} \mathcal{A}^2_{\lambda+k+\ell}(\mathbb{B}^{n-\ell})\big{)}, \\
 && \mathbf{T}^{\lambda} \asymp \bigoplus_{\rho \in \mathbb{Z}_+^m} I \otimes T^{\lambda+|\rho|+\ell} \ \longmapsto \ T^{\lambda+k +\ell}
\end{eqnarray*}
is a morphism (representation) of the $C^*$-algebra $\boldsymbol{\mathcal{T}}_{\lambda}(1,{\rm BUC}(\mathbb{B}^{n-\ell}))$. Thus the algebra $\mathcal{T}^{(n-\ell)}_{\mu}({\rm BUC}(\mathbb{B}^{n-\ell}))$, being its image, is norm closed. To finish the proof it is sufficient to mention that both algebras $\mathcal{T}^{(n-\ell)}_{\mu}({\rm BUC}(\mathbb{B}^{n-\ell}))$ and $\mathcal{T}_{\mu}({\rm BUC}(\mathbb{B}^{n-\ell}))$ share the same dense subalgebra, which consists of all elements of the form
\begin{equation*}
 \sum_{j =1}^k \prod_{p=1}^{ q_j} T^{\mu}_{c_{j,p}}, \quad \textrm{where}  \ \ c_{j,p} \in {\rm BUC}(\mathbb{B}^{n-\ell}) \ \ \textrm{and}  \ \ k,\,q_j \in \mathbb{N},
\end{equation*} 
being the images under the mapping $\iota_k$, of the elements
\begin{equation*}
 \sum_{j =1}^k \prod_{p=1}^{q_j} \mathbf{T}^{\lambda}_{f_{c_{j,p}}}
\end{equation*}
of the algebra $\boldsymbol{\mathcal{T}}_{\lambda}(1,{\rm BUC}(\mathbb{B}^{n-\ell}))$.
\end{proof}

\begin{cor} \label{co:irred-mu}
 The $C^*$-algebra $\mathcal{T}^{(n-\ell)}_{\mu}({\rm BUC}(\mathbb{B}^{n-\ell})) =\mathcal{T}_{\mu}({\rm BUC}(\mathbb{B}^{n-\ell}))$ is irreducible and contains the ideal $\mathcal{K}(\mathcal{A}^2_{\mu}(\mathbb{B}^{n-\ell}))$ of all compact operators on $\mathcal{A}^2_{\mu}(\mathbb{B}^{n-\ell})$.
\end{cor}
Recall that the subalgebra $\mathcal{T}_{\mu}(C(\overline{\mathbb{B}}^{n-\ell}))$ of $\mathcal{T}_{\mu}({\rm BUC}(\mathbb{B}^{n-\ell}))$ already contains $\mathcal{K}(\mathcal{A}^2_{\mu}(\mathbb{B}^{n-\ell}))$. We now describe (some of) the irreducible representations of the algebra $\boldsymbol{\mathcal{T}}_{\lambda}(1,{\rm BUC}(\mathbb{B}^{n-\ell}))$. 

\vspace{1mm}
\noindent{\bf 1.  Infinite dimensional representations:} By Corollary \ref{co:irred-mu}, for each $k \in \mathbb{Z}_+$ the representation
\begin{eqnarray} \label{eq:iota_k}
\iota_k &:& \boldsymbol{\mathcal{T}}_{\lambda}(1,{\rm BUC}(\mathbb{B}^{n-\ell})) \ \longrightarrow \ \mathcal{L}\big{(} \mathcal{A}^2_{\lambda+k+\ell}(\mathbb{B}^{n-\ell})\big{)} \\ \nonumber
 && \mathbf{T}^{\lambda} \asymp \bigoplus_{\rho \in \mathbb{Z}_+^m} I \otimes (T_c^{\lambda+|\rho| +\ell} + N^{\lambda+|\rho|+\ell}) \ \longmapsto \ T_c^{\lambda+k +\ell} + N^{\lambda+k+\ell},
\end{eqnarray}
whose image is $\mathcal{T}_{\lambda+k+\ell}({\rm BUC}(\mathbb{B}^{n-\ell}))$, is irreducible.

\medskip

We note that  the representations $\{ \iota_k\}_{k\in \mathbb{Z}_+}$ are not pairwise unitary equivalent. To see this we fix $k \in \mathbb{Z}_+$ and consider 
\begin{equation}\label{iota_k_radial}
\iota_k\big{(} {\bf T}_{f_{1-|z|^2}}^{\lambda}\big{)}= \iota_k \Big{(} \bigoplus_{\rho \in \mathbb{Z}_+^m} I \otimes T_{1-|z|^2}^{\lambda+|\rho|+\ell} \Big{)}
= T_{1-|z|^2}^{\lambda+k+\ell} \in \mathcal{L} \big{(} \mathcal{A}^2_{\lambda+k+\ell}(\mathbb{B}^{n-\ell}) \big{)}. 
\end{equation}
\par
Since the symbol $a(z):= 1-|z|^2$ is radial, the Toeplitz operator $T_{1-|z|^2}^{\lambda+k+\ell}$ is diagonal with respect to the standard monomial orthonormal basis 
of $\mathcal{A}^2_{\lambda+k+1}(\mathbb{B}^{n-\ell})$.
By \cite[Corollary~3.2]{GKW} the norm of this operator is given by
\begin{align*}
\|T_{1 - |z|^2}^{\lambda+k+\ell}\| &= \sup\limits_{m \in \mathbb{Z}_+} \frac{1}{B(m+n-\ell,\lambda+k+\ell+1)} \int_0^1 (1-r)^{\lambda+k+\ell+1}r^{m+n-\ell-1} \, dr\\
&= \sup\limits_{m \in \mathbb{Z}_+} \frac{1}{B(m+n-\ell,\lambda+k+\ell+1)}B(m+n-\ell,\lambda+k+\ell+2)\\
&= \sup\limits_{m \in \mathbb{Z}_+} \frac{\lambda+k+\ell+1}{m+n+\lambda+k+1} = \frac{\lambda+k+\ell+1}{n+\lambda+k+1},
\end{align*}
where $B(x,y) = \frac{\Gamma(x)\Gamma(y)}{\Gamma(x+y)}$ denotes the Beta function. \label{page:not} Since $n >\ell$ the norm of $T_{1 - |z|^2}^{\lambda+k+\ell}$ depends on $k$. 
Hence the representations $\iota_{k_1}$ and $\iota_{k_2}$ are not unitarily equivalent for $k_1 \neq k_2$.
\vspace{1ex}\\
{\bf 2. One-dimensional representations via a "quantization effect":} We consider the Berezin transform $\mathcal{B}_{\mu}$ for bounded operators on $\mathcal{A}_{\mu}^2(\mathbb{B}^{n-\ell})$.
In particular, we may consider a Toeplitz operator $T=T^{\mu}_c$ with symbol $c \in {\rm BUC}(\mathbb{B}^{n-\ell})$. 
By \cite[Lemma 4.5]{MiSuWi} the function $\mathcal{B}_{\mu}[c]$ belongs to  ${\rm BUC}(\mathbb{B}^{n-\ell})$ and thus admits a continuous extension to the compact set $M(\textup{BUC}):=M({\rm BUC}(\mathbb{B}^{n-\ell}))$ of maximal ideals of the $C^*$-algebra ${\rm BUC}(\mathbb{B}^{n-\ell})$.
\begin{lem} \label{le:rho}
The map $\nu \, : \ \boldsymbol{\mathcal{T}}_{\lambda}(1,{\rm BUC}(\mathbb{B}^{n-\ell})) \rightarrow \ {\rm BUC}(\mathbb{B}^{n-\ell}) = C(M(\textup{BUC}))$, defined by
\begin{equation*}
  \nu \, : \ \mathbf{T}^{\lambda} \asymp \bigoplus_{\rho \in \mathbb{Z}_+^m} I \otimes T^{\lambda+|\rho|+\ell} \ \ \longmapsto \ \ \lim_{|\rho|\to \infty} \mathcal{B}_{\lambda+|\rho|+\ell}(T^{\lambda+|\rho|+\ell})
\end{equation*}
is a continuous $*$-homomorphism of the $C^*$-algebra $\boldsymbol{\mathcal{T}}_{\lambda}(1,{\rm BUC}(\mathbb{B}^{n-\ell}))$ onto $C(M(\textup{BUC}))$. 
\end{lem}

\begin{proof}
By Theorem \ref{th:rep} all operators $T^{\lambda+|\rho|+\ell}$ are of the form $T_c^{\lambda+|\rho|+\ell} + N^{\lambda+|\rho|+\ell}$ with a common function 
$c \in {\rm BUC}(\mathbb{B}^{n-\ell})$ and operators $N^{\lambda+|\rho|+\ell}$ such that  $\|N^{\lambda+|\rho|+\ell}\| \to 0$ as $|\rho| \to \infty$. From 
$$|\mathcal{B}_{\lambda+|\rho|+\ell}(N^{\lambda+|\rho|+\ell})| \leq \|N^{\lambda+|\rho|+\ell}\|$$
it follows that $\lim_{|\rho| \to \infty} \mathcal{B}_{\lambda+|\rho|+\ell}(N^{\lambda+|\rho|+\ell}) = 0$. 
Then,  \cite[Proposition 4.4]{BC1} implies that 
$$\lim_{|\rho|\to \infty} \mathcal{B}_{\lambda+|\rho|+\ell}(T^{\lambda+|\rho|+\ell})(z) = c(z)$$ 
uniformly on $\mathbb{B}^{n-\ell}$. That is,
\begin{align*}
& \lim_{|\rho|\to \infty} \mathcal{B}_{\lambda+|\rho|+\ell}(T^{\lambda+|\rho|+\ell}) = \lim_{|\rho|\to \infty} \mathcal{B}_{\lambda+|\rho|+\ell}(T_c^{\lambda+|\rho|+\ell} + N^{\lambda+|\rho|+\ell}) \\
& = \lim_{|\rho|\to \infty} \mathcal{B}_{\lambda+|\rho|+\ell}(T_c^{\lambda+|\rho|+\ell})
 = c \in {\rm BUC}(\mathbb{B}^{n-\ell}) = C(M(\textup{BUC})),
\end{align*}
and $\nu$ is well-defined. It is also onto as for each $c \in {\rm BUC}(\mathbb{B}^{n-\ell})$ we have that $\nu(\boldsymbol{T}_{f_c}^{\lambda}) = c$. 
Moreover, the mapping $\nu$ is obviously a $*$-homomorphism, and its continuity follows from
\begin{align*}
\|c\|_{\infty} &= \sup_{z_2 \in \mathbb{B}^{n-\ell}} |c(z_2)| = \lim_{|\rho|\to \infty} \|T_c^{\lambda+|\rho|+\ell}\| \\
&\leq \sup_{|\rho| \in \mathbb{Z}_+}\| T_c^{\lambda+|\rho|+\ell} + N^{\lambda+|\rho|+\ell} \| = \|\mathbf{T}^{\lambda}\|. \qedhere
\end{align*}
\end{proof}

\begin{cor} \label{co:rho_z}
For each $\eta \in M(\textup{BUC})$, the map $\nu_{\eta}  :  \boldsymbol{\mathcal{T}}_{\lambda}(1,{\rm BUC}(\mathbb{B}^{n-\ell})) \rightarrow \mathbb{C}$, defined by
\begin{equation*} 
  \nu_{\eta} \, : \ \mathbf{T}^{\lambda} \ \ \longmapsto \ \ \nu(\mathbf{T}^{\lambda}) = c \ \ \longmapsto \ \ c(\eta) \in \mathbb{C},
\end{equation*}
is a one-dimensional representation of the $C^*$-algebra $\boldsymbol{\mathcal{T}}_{\lambda}(1,{\rm BUC}(\mathbb{B}^{n-\ell}))$.
\end{cor}
Let us  compare the above result with Remark \ref{rem:no-one-dim}. The $C^*$-algebras $\mathcal{T}_{\mu}({\rm BUC}(\mathbb{B}^{n-\ell}))$ and 
$\boldsymbol{\mathcal{T}}_{\lambda}(1,{\rm BUC}(\mathbb{B}^{n-\ell}))$ are generated by all Toeplitz operators $T^{\mu}_c$ and 
$\mathbf{T}^{\lambda}_{f_c}$, with $c \in {\rm BUC}(\mathbb{B}^{n-\ell})$, respectively. At the same time, the first algebra does not have non-trivial 
one-dimensional representations while the second one does have a lot of them.

\begin{rem} \label{rem:rec}
Given an operator $\boldsymbol{T}^{\lambda} \asymp \bigoplus_{\rho \in \mathbb{Z}_+^m} I \otimes T^{\lambda+|\rho|+\ell} \in \boldsymbol{\mathcal{T}}_{\lambda}(1,{\rm BUC}(\mathbb{B}^{n-\ell}))$, 
the result of Lemma~\ref{le:rho} permits us to recover its unique (by Theorem \ref{th:rep}) representation (\ref{eq:unique}). Indeed, all necessary data 
for the representation (\ref{eq:unique}) are given by
\begin{eqnarray*}
 c &=& \nu(\boldsymbol{T}^{\lambda}) = \lim_{|\rho|\to \infty} \mathcal{B}_{\lambda+|\rho|+\ell}(T^{\lambda+|\rho|+\ell}) \in {\rm BUC}(\mathbb{B}^{n-\ell}), \\
 N^{\lambda+|\rho|+\ell} &=& T^{\lambda+|\rho|+\ell} - T_c^{\lambda+|\rho|+\ell}.
\end{eqnarray*}
\end{rem}
\section{Symbols of vanishing oscillation at the boundary}
\label{se:VO}
\setcounter{equation}{0}
In the present section we consider the subalgebra  $\boldsymbol{\mathcal{T}}_{\lambda}(1,{\rm VO}_{\partial}(\mathbb{B}^{n-\ell}))$ of \\ $\boldsymbol{\mathcal{T}}_{\lambda}(1,{\rm BUC}(\mathbb{B}^{n-\ell}))$, which is generated by all Toeplitz operators $\boldsymbol{T}_{f_c}^{\lambda}$ acting on $\mathcal{A}_{\lambda}^2(\mathbb{B}^n)$ with $c \in {\rm VO}_{\partial}(\mathbb{B}^{n-\ell})$. Presenting the results, we will follow the lines of \cite[Section 4]{BV2}. Recall that
\begin{equation*}
 {\rm VO}_{\partial}(\mathbb{B}^{n-\ell}) = {\rm BUC}(\mathbb{B}^{n-\ell}) \cap{\rm VMO}_{\partial}(\mathbb{B}^{n-\ell})
\end{equation*}
is the maximal $C^*$-subalgebra of ${\rm BUC}(\mathbb{B}^{n-\ell})$ possessing the compact semi-commutator property: {\it for all $c_1,\, c_2 \in   
{\rm VO}_{\partial}(\mathbb{B}^{n-\ell})$ the semi-commutator $T^{\mu}_{c_1} T^{\mu}_{c_2} - T^{\mu}_{c_1c_2}$ is compact.} This leads to the following version of Theorem \ref{th:rep}:

\begin{thm} \label{th:rep_VO}
Each operator $\mathbf{T}^{\lambda} \in \boldsymbol{\mathcal{T}}_{\lambda}(1,{\rm VO}_{\partial}(\mathbb{B}^{n-\ell}))$, in the direct sum decomposition {\rm (\ref{eq:T_dir-sum})}, admits the unique representation
\begin{equation} \label{eq:unique:VO}
\mathbf{T}^{\lambda} \asymp \bigoplus_{\rho \in \mathbb{Z}_+^m} I \otimes (T_c^{\lambda+|\rho|+\ell} + K_{|\rho|}),
\end{equation}
where $c \in {\rm VO}_{\partial}(\mathbb{B}^{n-\ell})$, each $K_{|\rho|}$ is compact, and $\|K_{|\rho|}\| \to 0$ as $|\rho| \to \infty$.
\end{thm}

We thus obtain the following family of infinite dimensional irreducible representations $\iota_{\rho}$:

\begin{prop} \label{iota_rho_VO}
 For each $k \in \mathbb{Z}_+$ the restriction of the representation $\iota_k$ in \eqref{eq:iota_k} to the $C^*$-algebra $\boldsymbol{\mathcal{T}}_{\lambda}(1,{\rm VO}_{\partial}(\mathbb{B}^{n-\ell}))$ defines an irreducible representation on $\mathcal{A}^2_{\lambda+|\rho|+\ell}(\mathbb{B}^{n-\ell})$. Moreover,
\begin{equation} \label{eq:iota_k=}
 \iota_{k}(\boldsymbol{\mathcal{T}}_{\lambda}(1,{\rm VO}_{\partial}(\mathbb{B}^{n-\ell}))) = \mathcal{T}_{\lambda+k+\ell}({\rm VO}_{\partial}(\mathbb{B}^{n-\ell}))
\end{equation}
and 
\begin{equation} \label{eq:T_lambda+k+ell}
 \mathcal{T}_{\lambda+k+\ell}({\rm VO}_{\partial}(\mathbb{B}^{n-\ell})) = \{T_c^{\lambda+k+\ell} + K : \ c \in {\rm VO}_{\partial}(\mathbb{B}^{n-\ell}), \ \  K \ \mathrm{is \ compact}\}.
\end{equation}
Furthermore, the representations $\iota_{k_1}$ and $\iota_{k_2}$ are not unitarily equivalent for $k_1 \neq k_2$.
\end{prop}

\begin{proof}
 The proof of (\ref{eq:iota_k=}) literally follows the arguments of the proof of Proposition \ref{prop:coincide}. The algebra $\mathcal{T}_{\lambda+k+\ell}(C(\overline{\mathbb{B}}^{n-\ell}))$ is a subalgebra of $\mathcal{T}_{\lambda+k+\ell}({\rm VO}_{\partial}(\mathbb{B}^{n-\ell}))$ and, as it was already mentioned, contains the ideal $\mathcal{K}(\mathcal{A}^2_{\mu}(\mathbb{B}^{n-\ell}))$. 
 Thus 
 $$\mathcal{K}(\mathcal{A}^2_{\mu}(\mathbb{B}^{n-\ell})) \subset \mathcal{T}_{\lambda+k+\ell}({\rm VO}_{\partial}(\mathbb{B}^{n-\ell})),$$
 and the representation $\iota_k$ of the algebra $\mathcal{T}_{\lambda+k+\ell}({\rm VO}_{\partial}(\mathbb{B}^{n-\ell}))$ is irreducible.
Description (\ref{eq:T_lambda+k+ell}) follows then from (\ref{eq:unique:VO}), (\ref{eq:iota_k=}), and the inclusion \\ $\mathcal{K}(\mathcal{A}^2_{\mu}(\mathbb{B}^{n-\ell})) \subset \mathcal{T}_{\lambda+k+\ell}({\rm VO}_{\partial}(\mathbb{B}^{n-\ell}))$.

Finally, as $a(z)= 1-|z|^2$ belongs to ${\rm VO}_{\partial}(\mathbb{B}^{n-\ell})$, the arguments below \eqref{iota_k_radial} show that the representations $\iota_{k_1}$ and $\iota_{k_2}$ are not unitarily equivalent for $k_1 \neq k_2$.
\end{proof}
Let $M(\textup{VO})$ denote the maximal ideal space of ${\rm VO}_{\partial}(\mathbb{B}^{n-\ell})$. Note that $M(\textup{VO})$ densely contains $\mathbb{B}^{n-\ell}$ via evaluation maps, and let $M_{\partial}:=M(\textup{VO}) \setminus \mathbb{B}^{n-\ell}$.
It is easy to check (see \cite{BCZ} for the unweighted case) that 
\begin{equation*}
\mathcal{T}_{\lambda+|\rho|+\ell}({\rm VO}_{\partial}(\mathbb{B}^{n-\ell})) / \mathcal{K}(\mathcal{A}^2_{\lambda+|\rho|+\ell}(\mathbb{B}^{n-\ell})) \cong {\rm VO}_{\partial}(\mathbb{B}^{n-\ell}) / C_0(\mathbb{B}^{n-\ell}) \cong C(M_{\partial})
\end{equation*}
via
\[T_c^{\lambda+|\rho|+\ell} + \mathcal{K}(\mathcal{A}^2_{\lambda+|\rho|+\ell}(\mathbb{B}^{n-\ell})) \ \longmapsto  \ c + C_0(\mathbb{B}^{n-\ell}) 
 \ \cong \ c|_{M_{\partial}}.\]
This leads to the following family of one-dimensional irreducible representations:

\begin{cor}
For every $(\eta,k) \in M_{\partial} \times \mathbb{Z}_+$ the map
\[\pi_{\eta,k} : \mathbf{T}^{\lambda} \asymp \bigoplus_{\rho \in \mathbb{Z}_+^m} I \otimes (T_c^{\lambda+|\rho|+\ell} + K_{|\rho|}) \ \stackrel{ \iota_k}{\longmapsto} \ T_c^{\lambda+k+\ell} 
+ K_{k} \ \longmapsto \ c(\eta)\]
defines a one-dimensional irreducible representation. Moreover,  $\pi_{\eta_1,k}$ and $\pi_{\eta_2,k}$ are unitarily equivalent if and only if $\eta_1 = \eta_2$.
\end{cor}

Furthermore, we have the following versions of Lemma \ref{le:rho} and Corollary \ref{co:rho_z}:

\begin{lem} \label{le:rho_VO}
The mapping $\nu  : \boldsymbol{\mathcal{T}}_{\lambda}(1,{\rm VO}_{\partial}(\mathbb{B}^{n-\ell}))  \rightarrow  {\rm VO}_{\partial}(\mathbb{B}^{n-\ell}) = C(M(\textup{VO}))$ defined by
\begin{equation*}
  \nu \, : \ \mathbf{T}^{\lambda} \asymp \bigoplus_{\rho \in \mathbb{Z}_+^m} I \otimes T^{\lambda+|\rho|+\ell} \ \ \longmapsto \ \ \lim_{|\rho|\to \infty} \mathcal{B}_{\lambda+|\rho|+\ell}(T^{\lambda+|\rho|+\ell})
\end{equation*}
is a continuous $*$-homomorphism of the $C^*$-algebra $\boldsymbol{\mathcal{T}}_{\lambda}(1,{\rm VO}_{\partial}(\mathbb{B}^{n-\ell}))$ onto $C(M(\textup{VO}))$.
\end{lem}

\begin{cor} \label{co:rho_z_VO}
For each $\eta \in M(\textup{VO})$ the map $\nu_{\eta} \, : \ \boldsymbol{\mathcal{T}}_{\lambda}(1,{\rm VO}_{\partial}(\mathbb{B}^{n-\ell})) \longrightarrow \mathbb{C}$ defined by
\begin{equation*} 
  \nu_{\eta} \, : \ \mathbf{T}^{\lambda} \ \ \longmapsto \ \ \nu(\mathbf{T}^{\lambda}) = c \ \ \longmapsto \ \ c(\eta) \in \mathbb{C},
\end{equation*}
is a one-dimensional representation of the $C^*$-algebra $\boldsymbol{\mathcal{T}}_{\lambda}(1,{\rm VO}_{\partial}(\mathbb{B}^{n-\ell}))$.
\end{cor}

We want to show now that these are all (up to the unitary equivalence) irreducible representations of $\boldsymbol{\mathcal{T}}_{\lambda}(1,{\rm VO}_{\partial}(\mathbb{B}^{n-\ell}))$.
\vspace{1ex}\par 
Denote by $\boldsymbol{\mathcal{K}}(\mathcal{A}_{\lambda}^2(\mathbb{B}^n))$ the set of all compact operators in $\mathcal{L}(\mathcal{A}_{\lambda}^2(\mathbb{B}^n))$. Moreover, put 
\[\boldsymbol{\mathcal{K}}_{\lambda}({\rm VO}_{\partial}(\mathbb{B}^{n-\ell})) := \{\mathbf{K}^{\lambda} \in \boldsymbol{\mathcal{T}}_{\lambda}(1,{\rm VO}_{\partial}(\mathbb{B}^{n-\ell})) : \nu(\mathbf{K}^{\lambda}) = 0\}.\]
In the representation \eqref{eq:unique:VO} these are exactly the operators of the form
\begin{equation} \label{eq:compacts}
\mathbf{K}^{\lambda} \asymp \bigoplus_{\rho \in \mathbb{Z}_+^m} I \otimes K_{|\rho|}
\end{equation}
with $K_{|\rho|}$ compact and $\|K_{|\rho|}\| \to 0$ as $|\rho| \to \infty$ (cf.~Remark \ref{rem:rec}). 

\begin{lem} \label{le:comp_VO}
We have $\boldsymbol{\mathcal{T}}_{\lambda}(1,{\rm VO}_{\partial}(\mathbb{B}^{n-\ell})) \cap 
\boldsymbol{\mathcal{K}}(\mathcal{A}_{\lambda}^2(\mathbb{B}^n)) = \boldsymbol{\mathcal{K}}_{\lambda}({\rm VO}_{\partial}(\mathbb{B}^{n-\ell}))$.
\end{lem}

\begin{proof}
Let $\mathbf{K}^{\lambda} \asymp \bigoplus_{\rho \in \mathbb{Z}_+^m} I \otimes K_{|\rho|} \in \boldsymbol{\mathcal{K}}_{\lambda}({\rm VO}_{\partial}(\mathbb{B}^{n-\ell}))$. \\
Since $K^{\rho} \in \mathcal{K}(\mathcal{A}^2_{\lambda+|\rho|+\ell}(\mathbb{B}^{n-\ell}))$ and the spaces $\mathscr{H}_{\rho}$ in (\ref{defn_mathcal_H}) are finite-dimensional, 
the operators $\bigoplus\limits_{|\rho| \leq k} I \otimes K_{|\rho|}$ are compact for all $k \in \N$. Moreover,
\[\|\mathbf{K}^{\lambda} -  \bigoplus\limits_{|\rho| \leq k} I \otimes K_{|\rho|}\| \leq \sup\limits_{|\rho| > k} \|I \otimes K_{|\rho|}\| = \sup\limits_{|\rho| > k} \|K_{|\rho|}\| \to 0\]
as $k \to \infty$. This implies that $\mathbf{K}^{\lambda}$ is indeed compact.
\vspace{1ex}\par 
Now let $\mathbf{T}^{\lambda} \asymp \bigoplus\limits_{\rho \in \mathbb{Z}_+^m} I \otimes T^{\lambda+|\rho|+\ell} \in \boldsymbol{\mathcal{T}}_{\lambda}(1,{\rm VO}_{\partial}(\mathbb{B}^{n-\ell}))$ be compact. 
Then  $\|T^{\lambda+|\rho|+\ell}\| \to 0$ as $|\rho| \to \infty$ and therefore 
\[|\nu(\mathbf{T}^{\lambda})| = \lim\limits_{|\rho|\to \infty} |\mathcal{B}_{\lambda+|\rho|+\ell}(T^{\lambda+|\rho|+\ell})| \leq \lim_{|\rho|\to \infty} \|T^{\lambda+|\rho|+\ell}\| = 0,\]
and hence $\mathbf{T}^{\lambda} \in \boldsymbol{\mathcal{K}}_{\lambda}({\rm VO}_{\partial}(\mathbb{B}^{n-\ell}))$.
\end{proof}\noindent 
This implies that the Calkin algebra $\boldsymbol{\mathcal{T}}_{\lambda}(1,{\rm VO}_{\partial}(\mathbb{B}^{n-\ell})) / (\boldsymbol{\mathcal{T}}_{\lambda}(1,{\rm VO}_{\partial}(\mathbb{B}^{n-\ell})) \cap \boldsymbol{\mathcal{K}}(\mathcal{A}_{\lambda}^2(\mathbb{B}^n)))$ is isomorphic to ${\rm VO}_{\partial}(\mathbb{B}^{n-\ell})$ via the induced mapping
\[\hat{\nu} \, : \ \mathbf{T}^{\lambda} + \boldsymbol{\mathcal{T}}_{\lambda}(1,{\rm VO}_{\partial}(\mathbb{B}^{n-\ell})) \cap \boldsymbol{\mathcal{K}}(\mathcal{A}_{\lambda}^2(\mathbb{B}^n))
 \ \ \longmapsto \ \ \lim_{|\rho|\to \infty} \mathcal{B}_{\lambda+|\rho|+\ell}(T^{\lambda+|\rho|+\ell}).\]

\begin{thm} \label{th:rep_list}
Up to unitary equivalence the following list of irreducible representations of the $C^*$-algebra $\boldsymbol{\mathcal{T}}_{\lambda}(1,{\rm VO}_{\partial}(\mathbb{B}^{n-\ell}))$ is complete: \begin{itemize}
	\item[\textup{(1)}] $\iota_k : \mathbf{T}^{\lambda} \asymp \bigoplus_{\rho \in \mathbb{Z}_+^m} I \otimes (T_c^{\lambda+|\rho|+\ell} + K_{|\rho|}) \mapsto 
	T_c^{\lambda+k+\ell} + K_k$ for $k \in \mathbb{Z}_+$,
	\item[\textup{(2)}] $\nu_{\eta} \, : \ \mathbf{T}^{\lambda} \ \ \longmapsto \ \ \nu(\mathbf{T}^{\lambda}) = c \ \ \longmapsto \ \ c(\eta) \in \mathbb{C}$ for 
	$\eta \in M(\textup{VO})$.   
	\end{itemize}
Moreover, the above representations are pairwise not unitarily equivalent.
\end{thm}
\begin{proof}
By Proposition \ref{iota_rho_VO} and Corollary \ref{co:rho_z_VO} it remains to show that the list is complete. \cite[Proposition 2.11.2]{Di} implies that every irreducible representation of 
$\boldsymbol{\mathcal{T}}_{\lambda}(1,{\rm VO}_{\partial}(\mathbb{B}^{n-\ell}))$ is either induced by a representation of 
$$\boldsymbol{\mathcal{T}}_{\lambda}(1,{\rm VO}_{\partial}(\mathbb{B}^{n-\ell})) / (\boldsymbol{\mathcal{T}}_{\lambda}(1,{\rm VO}_{\partial}(\mathbb{B}^{n-\ell})) \cap \boldsymbol{\mathcal{K}}(\mathcal{A}_{\lambda}^2(\mathbb{B}^n))) \cong {\rm VO}_{\partial}(\mathbb{B}^{n-\ell})$$ 
or is the extension of an irreducible representation of 
$\boldsymbol{\mathcal{T}}_{\lambda}(1,{\rm VO}_{\partial}(\mathbb{B}^{n-\ell})) \cap \boldsymbol{\mathcal{K}}(\mathcal{A}_{\lambda}^2(\mathbb{B}^n))$. The former representations are exactly the representations $\nu_{\eta}$. So consider an irreducible representation of 
$\boldsymbol{\mathcal{T}}_{\lambda}(1,{\rm VO}_{\partial}(\mathbb{B}^{n-\ell})) \cap \boldsymbol{\mathcal{K}}(\mathcal{A}_{\lambda}
 ^2(\mathbb{B}^n)) = \boldsymbol{\mathcal{K}}_{\lambda}({\rm VO}_{\partial}(\mathbb{B}^{n-\ell}))$. Restricted to the different levels $\rho$, this representation is either $0$ or an irreducible representation of $\mathcal{K}(\mathcal{A}^2_{\lambda+|\rho|+\ell}(\mathbb{B}^{n-\ell}))$ (cf.~Proposition \ref{iota_rho_VO}). Since the only irreducible representation of $\mathcal{K}(\mathcal{A}^2_{\lambda+|\rho|+\ell}(\mathbb{B}^{n-\ell}))$ is the identical representation and $\|K_{|\rho|}\| \to 0$ in \eqref{eq:compacts}, we can deduce that $\iota_k$, with $k =|\rho|$, are the only additional representations coming from $\boldsymbol{\mathcal{T}}_{\lambda}(1,{\rm VO}_{\partial}(\mathbb{B}^{n-\ell})) \cap \boldsymbol{\mathcal{K}}(\mathcal{A}_{\lambda}^2(\mathbb{B}^n))$.
\end{proof}

As a corollary of Lemma \ref{le:rho_VO} and Lemma \ref{le:comp_VO}, we obtain a result on the Fredholmness of an operator $\mathbf{T}^{\lambda} \in \boldsymbol{\mathcal{T}}_{\lambda}(1,{\rm VO}_{\partial}(\mathbb{B}^{n-\ell}))$.

\begin{cor} \label{co:Fred}
An operator $$\mathbf{T}^{\lambda} \asymp \bigoplus_{\rho \in \mathbb{Z}_+^m} I \otimes (T_c^{\lambda+|\rho|+\ell} + K_{|\rho|}) \in \boldsymbol{\mathcal{T}}_{\lambda}(1,{\rm VO}_{\partial}(\mathbb{B}^{n-\ell}))$$ 
is Fredholm if and only if $c(\eta) \neq 0$ for all $\eta \in M(\textup{VO})$. In particular, $\esssp \mathbf{T}^{\lambda} = c(M(\textup{VO}))$.
\end{cor}

Without presenting much details we briefly discuss the Fredholm index in the more general situation of matrix-valued symbols. As is shown in \cite[Section 2]{XZ}, the index formula for Fredholm operators with $\mathrm{VO}_{\partial}$-symbols can be reduced to the case of operator symbols that are continuous up to the boundary. More precisely, given a matrix-valued function 
$$c \in \mathrm{Mat}_p(\mathrm{VO}_{\partial}(\mathbb{B}^{n-\ell})):= \mathrm{VO}_{\partial}(\mathbb{B}^{n-\ell}) \otimes \mathrm{Mat}_p(\mathbb{C})$$ 
we define
\begin{equation} \label{eq:c_s}
 c_s(r\zeta) = \begin{cases}
                c(r\zeta), & \mathrm{if} \ \ 0 \leq r \leq s, \\
                c(s\zeta), & \mathrm{if} \ \ s < r \leq 1,
               \end{cases}
\end{equation}
for each $s \in (0,1)$, where $r = |z|$ and $\zeta \in S^{2(n - \ell)-1}= \partial \mathbb{B}^{n-\ell}$. Then $c_s(r\zeta) \in \mathrm{Mat}_p(C(\overline{\mathbb{B}}^{n-\ell}))$ for all $s \in (0,1)$.
The matrix-valued version of Corollary \ref{co:Fred} reads as follows:
\begin{cor}\label{Corollary_Fredholm_property_and_index}
Given $c \in \mathrm{Mat}_p(\mathrm{VO}_{\partial}(\mathbb{B}^{n-\ell}))$, the operator $$\mathbf{T}^{\lambda} \asymp \bigoplus_{\rho \in \mathbb{Z}_+^m} I \otimes (T_c^{\lambda+|\rho|+\ell} 
+ K_{|\rho|}) \in \boldsymbol{\mathcal{T}}_{\lambda}(1,{\rm VO}_{\partial}(\mathbb{B}^{n-\ell})) \otimes \mathrm{Mat}_p(\mathbb{C})$$ 
is Fredholm if and only if the matrix $c(\eta)$ is invertible for all $\eta \in M(\textup{VO})$. In particular, 
\begin{equation*}
 \esssp \mathbf{T}^{\lambda} = \big{\{} \det c(\eta)\, : \ \eta \in M(\textup{VO})\big{\}}.
\end{equation*}
In case of being Fredholm,
\begin{equation*}
 \Ind \mathbf{T}^{\lambda} = \Ind \left[\bigoplus_{\rho \in \mathbb{Z}_+^m} I \otimes (T_c^{\lambda+|\rho|+\ell} + K_{|\rho|})\right] = 0.
\end{equation*}
\end{cor}

\begin{proof}
 Only the index calculation needs to be justified. If $\mathbf{T}^{\lambda}$ is Fredholm, then the matrix-valued function $c$ is invertible, and
\begin{equation*}
 \Ind \mathbf{T}^{\lambda} = \sum_{\rho \in \mathbb{Z}_+^m} \dim \mathscr{H}_{\rho} \times \Ind (T_c^{\lambda+|\rho|+\ell} + K_{|\rho|}). 
\end{equation*}
For each $\rho \in \mathbb{Z}_+^m$ Theorem 2.6 of  \cite{XZ} ensures that there is $s_{|\rho|,0} \in (0,1)$ such that for every $s \in (s_{|\rho|,0},1)$, each operator 
$T_{c_s}^{\lambda+|\rho|+\ell} + K_{|\rho|}$ is Fredholm and 
\begin{equation*}
 \Ind (T_c^{\lambda+|\rho|+\ell} + K_{|\rho|}) = \Ind (T_{c_s}^{\lambda+|\rho|+\ell} + K_{|\rho|}),
\end{equation*}
where $c_s$ is defined in \eqref{eq:c_s}. The matrix-valued function $c_s$ is continuous on the closed unit ball $\overline{\mathbb{B}}^{n-\ell}$, which is retractable to a point. Hence $c_s$ is homotopic to a constant matrix, 
say $c_{|\rho|}$, in a class of invertible continuous matrix-functions on $\overline{\mathbb{B}}^{n-\ell}$. This implies that the operator $T_{c_s}^{\lambda+|\rho|+\ell} 
+ K_{|\rho|}$ is homotopic to the scalar-matrix multiplication operator $c_{|\rho|}$. Thus, for each $\rho \in \mathbb{Z}_+^m$,
\begin{equation*}
 \Ind (T_c^{\lambda+|\rho|+\ell} + K_{|\rho|}) = \Ind c_{|\rho|}I = 0. \qedhere
\end{equation*}
For a discussion ensuring that $\ker \mathbf{T}^{\lambda}$ and $\text{coker}\, \mathbf{T}^{\lambda}$ are finite dimensional see \cite[p. 730]{BV2}.
\end{proof}
\section{The algebra $\boldsymbol{\mathcal{T}}_{\lambda}(L^{\infty}_{k\textup{-}qr}, {\rm VO}_{\partial}(\mathbb{B}^{n-\ell}))$}
\setcounter{equation}{0}
\label{Section_6}
Given $m \leq \ell$, we fix a multi-index $k = (k_1,...,k_m) \in \mathbb{N}^m$ with $|k| = k_1 + ...+ k_m = \ell$, and introduce the algebra $L^{\infty}_{k\textup{-}qr}$ of \emph{$k$-quasi-radial} functions
\begin{equation*}
L^{\infty}_{k\textup{-}qr} = \big{\{} a(z') = a(|z'_{(1)}|, ..., |z'_{(m)}|) \in L^{\infty}(\mathbb{B}^{\ell})\: : \: a \in L^{\infty}(\tau( \mathbb{B}^m)) \big{\}},
\end{equation*}
where $ \tau( \mathbb{B}^m)=\{ r=(r_1, \ldots, r_m) \in \mathbb{R}_+^m \: : \: 0 \leq r <1\}$ is the base of the ball $\mathbb{B}^m$ considered as a \emph{Reinhardt domain}.
Note that in the extreme cases: $m=1$ so that $k = (\ell)$ and $m=\ell$ so that $k = (1,\ldots,1)$, we deal with \emph{radial} and \emph{separately radial} symbols, respectively. In the 
following we write $r_j:= |z'_{(j)}|$ for all $j = 1,...,m$ and $a=a(r_1,...,r_m) \in L^{\infty}(\tau(\mathbb{B}^m))$. 
\vspace{1ex}\par 
According to \cite[Lemma 3.1]{V1}, a Toeplitz operator $T^{\lambda}_a$ with a  $k$-quasi-radial symbol $a=a(r_1, \cdots,r_m)$ is diagonal with respect to $(e^{\lambda}_{\alpha'}(z'))_{\alpha^{\prime} 
\in \mathbb{Z}_+^{\ell}}$. For each multi-index $\alpha = (\alpha_1, \cdots, \alpha_{\ell}) = (\alpha_{(1)}, \cdots, \alpha_{(m)}) \in \mathbb{Z}^{\ell}_+$ we denote by $\rho = \rho(\alpha) = (\rho_1,\cdots,\rho_m) \in \mathbb{Z}_+^m$ the multi-index with the entries $\rho_j = |\alpha_{(j)}|$, for all $j=1,\cdots,m$. In particular, we have $|\alpha| = |\rho|$. The eigenvalue $\gamma_{a,k,\lambda}(\rho)$ of $T^{\lambda}_a$ with respect to $e^{\lambda}_{\alpha'}$ only depends on $\rho=\rho(\alpha)$, i.e.,
\begin{equation*}
T^{\lambda}_ae^{\lambda}_{\alpha'}=\gamma_{a,k,\lambda}(\rho) e^{\lambda}_{\alpha'}, 
\end{equation*}
where
\begin{equation*}
\gamma_{a,k,\lambda}(\rho)=\frac{2^m\Gamma(n+|\rho|+\lambda+1)}{\Gamma(\lambda+1)\prod_{j=1}^m(k_j-1+\rho_j)!}
\int_{\mathcal{\tau}(\mathbb{B}^m)}a(r)(1-|r|^2)^{\lambda}\prod_{j=1}^mr_j^{2\rho_j+2k_j-1}dr.
\end{equation*} 
\par 
In particular, the $C^*$-algebra $\mathcal{T}_{\lambda}(L_{k\textup{-}qr}^{\infty})$ is infinitely generated and unital. It consists of certain operators that are diagonal with respect to 
the orthonormal basis $[e^{\lambda}_{\alpha'}\: : \: \alpha' \in \mathbb{Z}_+^{\ell}]$. By identifying elements in  $\mathcal{T}_{\lambda}(L_{k\textup{-}qr}^{\infty})$ with its eigenvalue 
sequence, we can interpret $\mathcal{T}_{\lambda}(L_{k\textup{-}qr}^{\infty})$ as a sub-algebra of $l_{\infty}(\mathbb{Z}_+^m)$. In fact, the algebra 
$\mathcal{T}_{\lambda}(L_{k\textup{-}qr}^{\infty})$ can even be embedded into a smaller algebra of, in a certain sense, slowly oscillating  sequences (see \cite[Section 3]{V1} for 
details). We denote by $\mathrm{SO}(\mathbb{Z}_+^m)$ the image of $\mathcal{T}_{\lambda}(L_{k\textup{-}qr}^{\infty})$ under this identification. 
We denote the compact set of maximal ideals of $\mathrm{SO}(\mathbb{Z}_+^m)$ (coinciding with the compact of maximal ideals of  $\mathcal{T}_{\lambda}(L_{k\textup{-}qr}^{\infty})$) by 
$M(\textup{SO})$. Note that $M(\textup{SO})$ densely contains $\mathbb{Z}_+^m$  via evaluation maps, and let $M_{\infty}= M(\textup{SO})\setminus \mathbb{Z}_+^m$. Although this 
will not be used in the paper, we mention that the set $M_{\infty}$ admits a quite sophisticated fibration, whose description is contained in \cite[Section 3]{BV1}.

\medskip
By Corollary \ref{co:f_ac} (and Remark \ref{re:abrev}), for each $a \in L_{k\textup{-}qr}^{\infty}$ and $c \in {\rm VO}_{\partial}(\mathbb{B}^{n-\ell})$, we have
\begin{equation} \label{quasi-radial_VO_decomposition}
 {\bf T}^{\lambda}_{f_{ac}}  \asymp \bigoplus_{\rho \in \mathbb{Z}_+^m} \gamma_{a,k,\lambda}(\rho)I \otimes T^{\lambda+|\rho|+\ell}_c 
 = \bigoplus_{\rho \in \mathbb{Z}_+^m} I_{\mathscr{H}_{\rho}} \otimes T^{\lambda+|\rho|+\ell}_{\gamma_{a,k,\lambda}(\rho)\,c}.
\end{equation}
Here the operator $I_{\mathscr{H}_{\rho}}$ indicates that  $T^{\lambda+|\rho|+\ell}_{\gamma_{a,k,\lambda}(\rho)\,c}$ is taken with multiplicity $\dim \mathscr{H}_{\rho}$.
\begin{lem} \label{le:comp_k-qr_VO}
 The intersection $\boldsymbol{\mathcal{T}}_{\lambda}(L^{\infty}_{k\textup{-}qr}, {\rm VO}_{\partial}(\mathbb{B}^{n-\ell})) \cap \boldsymbol{\mathcal{K}}(\mathcal{A}^2_{\lambda}(\mathbb{B}^n))$ 
 consists of all operators of the form
\begin{equation*}
 \boldsymbol{K}^{\lambda} \asymp \bigoplus_{\rho \in \mathbb{Z}_+^m} I_{\mathscr{H}_{\rho}} \otimes K_{\rho},
\end{equation*}
where each $K_{\rho}$ is compact on $\mathcal{A}^2_{\lambda+|\rho|+\ell}(\mathbb{B}^{n-\ell})$ and $\|K_{\rho}\| \to 0$ as $|\rho| \to \infty$.
\end{lem}

\begin{proof}
By Equation \eqref{quasi-radial_VO_decomposition} every operator $\boldsymbol{K}^{\lambda}$ in $\boldsymbol{\mathcal{T}}_{\lambda}
(L^{\infty}_{k\textup{-}qr}, {\rm VO}_{\partial}(\mathbb{B}^{n-\ell}))$ has  the form
\[ \boldsymbol{K}^{\lambda} \asymp \bigoplus_{\rho \in \mathbb{Z}_+^m} I_{\mathscr{H}_{\rho}} \otimes K_{\rho}.\]
If $\boldsymbol{K}^{\lambda}$ is compact, it is easy to see that every $K_{\rho}$ has to be compact and $\|K_{\rho}\| \to 0$ as $|\rho| \to \infty$.

Conversely, for $\rho \in \mathbb{Z}_+^m$ let $K_{\rho} \in \mathcal{K}(\mathcal{A}^2_{\lambda+|\rho|+\ell}(\mathbb{B}^{n-\ell}))$ be arbitrary compact operators with $\|K_{\rho}\| \to 0$ as 
$|\rho| \to \infty$. Consider $\boldsymbol{K}_n^{\lambda} :\asymp \bigoplus_{|\rho| \leq n} I_{\mathscr{H}_{\rho}} \otimes K_{\rho}$ for $n \in \N$. Then the sequence 
$(\boldsymbol{K}_n^{\lambda})_{n \in \N}$ converges to $\boldsymbol{K}^{\lambda} \asymp \bigoplus_{\rho \in \mathbb{Z}_+^m} I_{\mathscr{H}_{\rho}} \otimes K_{\rho}$ as in the proof of 
Lemma \ref{le:comp_VO}. In particular, $\boldsymbol{K}^{\lambda}$ is compact. 
It remains to show that all operators $\boldsymbol{K}_n^{\lambda}$ belong to $\boldsymbol{\mathcal{T}}_{\lambda}(L^{\infty}_{k\textup{-}qr}, {\rm VO}_{\partial}(\mathbb{B}^{n-\ell}))$, or that 
for each $\rho \in \mathbb{Z}_+^m$ the operator
\begin{equation} \label{eq:T}
 {\bf T}^{\lambda}_{\rho} \asymp \bigoplus_{\kappa \in \mathbb{Z}_+^m} T_{\kappa}, \ \ \ \textrm{where} \ \ T_{\kappa} =
\begin{cases}
 0, &  \textrm{for} \ \ \kappa \neq \rho, \\
 I_{\mathscr{H}_{\rho}} \otimes K_{\rho}, &  \textrm{for} \ \ \kappa = \rho.                                                                                                                                                                                                            
\end{cases}
\end{equation}
belongs to $\boldsymbol{\mathcal{T}}_{\lambda}(L^{\infty}_{k\textup{-}qr}, {\rm VO}_{\partial}(\mathbb{B}^{n-\ell}))$.

By \cite[Corollary 3.3]{BV1} and Corollary \ref{co:f_ac}, there is an operator ${\bf P}^{\lambda}_{\rho} \in \boldsymbol{\mathcal{T}}_{\lambda}(L^{\infty}_{k\textup{-}qr}, 1)$
such that
\begin{equation*}
 {\bf P}^{\lambda}_{\rho} \asymp P_{\rho} \otimes I = \bigoplus_{\kappa \in \mathbb{Z}_+^m}
\delta_{\kappa,\rho} I_{\mathscr{H}_{\kappa}} \otimes I,
\end{equation*}
where $P_{\rho}$ is the orthogonal projection of $\mathcal{A}^2_{\lambda}(\mathbb{B}^{\ell})$ onto $\mathscr{H}_{\kappa}$. Then, by Proposition \ref{iota_rho_VO}, there is an operator ${\bf K}_{\rho} \in \boldsymbol{\mathcal{T}}_{\lambda}(1, {\rm VO}_{\partial}(\mathbb{B}^{n-\ell}))$
such that $\iota_{|\rho|}({\bf K}_{\rho}) = K_{\rho}$. Thus ${\bf P}^{\lambda}_{\rho}{\bf K}_{\rho}$ belongs to $\boldsymbol{\mathcal{T}}_{\lambda}(L^{\infty}_{k\textup{-}qr}, {\rm VO}_{\partial}(\mathbb{B}^{n-\ell}))$ and is exactly the operator in \eqref{eq:T}.
\end{proof}

Below we frequently consider the tensor product $\mathcal{A} \otimes \mathcal{B}$ of two commutative $C^*$-algebras $\mathcal{A}$ and $\mathcal{B}$. Recall that as $\mathcal{A}$ and $\mathcal{B}$ are commutative (and thus nuclear), the $C^*$-norm on $\mathcal{A} \otimes \mathcal{B}$ is uniquely defined. In particular, if $M(\mathcal{A})$ and $M(\mathcal{B})$ are the (locally) compact sets of maximal ideals of $\mathcal{A}$ and $\mathcal{B}$, respectively, then 
\begin{equation} \label{tensor_product_maximal_ideals}
 \mathcal{A} \otimes \mathcal{B} \cong C(M(\mathcal{A})) \otimes C(M(\mathcal{B})) = 
 C(M(\mathcal{A}) \times M(\mathcal{B})).
\end{equation}
\par 
By $\mathcal{A} \otimes_a \mathcal{B}$ we denote the algebraic tensor product of $\mathcal{A}$ and $\mathcal{B}$, which consists of all finite sums of the form $\sum a_k \otimes b_k$, $a_k \in \mathcal{A}$ and $b_k \in \mathcal{B}$. Another auxiliary proposition is needed:

\begin{prop} \label{prop:Berezin_VO}
{\rm (see \cite[Theorem 25]{BCZ} for the unweighted case)}\\
Let $f \in {\rm VO}_{\partial}(\mathbb{B}^{n-\ell})$ and $\mu > -1$. Then
\[\lim_{|z| \to 1} |f(z) - \mathcal{B}_{\mu}[f](z)| = 0.\]
\end{prop}

\begin{proof}
Let $\varepsilon > 0$ and choose $r \in (0,1)$ sufficiently large such that $dv_{\mu}(r\mathbb{B}^{n-\ell}) > 1 - \frac{\varepsilon}{2\|f\|_{\infty}}$. Now choose $z$ sufficiently close to the boundary such 
that $|f(z) - f(w)| < \varepsilon$ for all $w \in \varphi_z^{-1}(r\mathbb{B}^{n-\ell})$, where $\varphi_z$ is the usual automorphism on $\mathbb{B}^{n-\ell}$ that interchanges $0$ and $z$. Then
\begin{align*}
|f(z) - \mathcal{B}_{\mu}[f](z)| &\leq \int_{\mathbb{B}^{n-\ell}} |f(z)-f(\varphi_z(w))| \, dv_{\mu}(w)\\
&= \int_{r\mathbb{B}^{n-\ell}} |f(z)-f(\varphi_z(w))| \, dv_{\mu}(w) \\
&+ \int_{\mathbb{B}^{n-\ell} \setminus r\mathbb{B}^{n-\ell}} |f(z)-f(\varphi_z(w))| \, dv_{\mu}(w)\\
&< \int_{r\mathbb{B}^{n-\ell}} \varepsilon \, dv_{\mu}(w) + \int_{\mathbb{B}^{n-\ell} \setminus r\mathbb{B}^{n-\ell}} 2\|f\|_{\infty} \, dv_{\mu}(w) < 2\varepsilon,
\end{align*}
and the proposition follows.
\end{proof}

Let $T_f$ be compact with symbol $f \in {\rm VO}_{\partial}(\mathbb{B}^{n-\ell})$. Then \cite[Theorem 5.5]{MiSuWi} implies that $\mathcal{B}_{\mu}(f)$ is contained in $C_0(\mathbb{B}^{n-\ell})$, the 
set of continuous functions on  $\mathbb{B}^{n-\ell}$ vanishing at the boundary $\partial\mathbb{B}^{n-\ell}$. Hence, by Proposition \ref{prop:Berezin_VO}, $f \in C_0(\mathbb{B}^{n-\ell})$. On the other hand, every Toeplitz operator with symbol in $C_0(\mathbb{B}^{n-\ell})$ is compact. Therefore $T_f$ with $f \in {\rm VO}_{\partial}(\mathbb{B}^{n-\ell})$ is compact if and only if $f$ belongs to $C_0(\mathbb{B}^{n-\ell})$. 
\begin{rem}{\rm 
 We remark that  the above statement remains true for compact Toeplitz operators with symbols in ${\rm BUC}(\mathbb{B}^{n-\ell})$ by \cite[Theorem 3.8]{BC2} 
(and even boundedness of the symbol is not required). }
\end{rem}

\begin{cor} \label{co:comp}
An operator $\boldsymbol{T}^{\lambda} \in \boldsymbol{\mathcal{T}}_{\lambda}(L^{\infty}_{k\textup{-}qr}, {\rm VO}_{\partial}(\mathbb{B}^{n-\ell}))$ of the form 
\begin{equation*}
 \boldsymbol{T}^{\lambda} \asymp \bigoplus_{\rho \in \mathbb{Z}_+^m} I_{\mathscr{H}_{\rho}} \otimes T_{d(\rho,z'')}^{\lambda+|\rho|+\ell}
\end{equation*}
with $d(\rho,z'') \in \mathrm{SO}(\mathbb{Z}_+^m) \otimes {\rm VO}_{\partial}(\mathbb{B}^{n-\ell})$ is compact if and only if $d(\rho,z'') \in c_0 \otimes C_0(\mathbb{B}^{n-\ell})$.
\end{cor}

We describe now a general form of elements from  $\boldsymbol{\mathcal{T}}_{\lambda}(L^{\infty}_{k\textup{-}qr}, {\rm VO}_{\partial}(\mathbb{B}^{n-\ell}))$.

\begin{thm} \label{th:big-repr}
 Each element $\boldsymbol{T}^{\lambda} \in \boldsymbol{\mathcal{T}}_{\lambda}(L^{\infty}_{k\textup{-}qr}, {\rm VO}_{\partial}(\mathbb{B}^{n-\ell}))$ admits the representation
\begin{equation} \label{eq:big-repr} 
 \boldsymbol{T}^{\lambda} \asymp \bigoplus_{\rho \in \mathbb{Z}_+^m} I_{\mathscr{H}_{\rho}} \otimes (T_{c(\rho,z'')}^{\lambda+|\rho|+\ell} + K_{\rho}),
\end{equation}
where $c = c(\rho,z'') \in \mathrm{SO}(\mathbb{Z}_+^m) \otimes {\rm VO}_{\partial}(\mathbb{B}^{n-\ell})$, the operators $K_{\rho}$ are compact for all $\rho \in \mathbb{Z}_+^m$ 
and $\|K_{\rho}\| \to 0$ as $|\rho| \to \infty$.
\end{thm}

\begin{proof}
By \eqref{quasi-radial_VO_decomposition}, the generators $\boldsymbol{T}^{\lambda}$ of the algebra $\boldsymbol{\mathcal{T}}_{\lambda}(
L^{\infty}_{k\textup{-}qr}, {\rm VO}_{\partial}(\mathbb{B}^{n-\ell}))$ take the form \eqref{eq:big-repr}. Now consider finite sums of finite products of generators. To show that these operators 
can again be represented in the form \eqref{eq:big-repr} it clearly suffices to consider finite products. For  $j = 1, \ldots, q$  let $\boldsymbol{T}_j^{\lambda} :\asymp 
\bigoplus_{\rho \in \mathbb{Z}_+^m} I_{\mathscr{H}_{\rho}} \otimes T^{\lambda+|\rho|+\ell}_{\gamma_{a_j,k,\lambda}(\rho)\,c_j}$ be generators. Then
\begin{align*}
\boldsymbol{T}_1^{\lambda} \cdots \boldsymbol{T}_q^{\lambda} &\asymp \bigoplus_{\rho \in \mathbb{Z}_+^m} I_{\mathscr{H}_{\rho}} \otimes 
(T^{\lambda+|\rho|+\ell}_{\gamma_{a_1,k,\lambda}(\rho)\,c_1} \cdots T^{\lambda+|\rho|+\ell}_{\gamma_{a_q,k,\lambda}(\rho)\,c_q})\\
&= \bigoplus_{\rho \in \mathbb{Z}_+^m} I_{\mathscr{H}_{\rho}} \otimes \gamma_{a_1,k,\lambda}(\rho) \cdots \gamma_{a_q,k,\lambda}(\rho)(T^{\lambda+|\rho|+\ell}_{c_1} 
\cdots T^{\lambda+|\rho|+\ell}_{c_q})\\
&= \bigoplus_{\rho \in \mathbb{Z}_+^m} I_{\mathscr{H}_{\rho}} \otimes (T^{\lambda+|\rho|+\ell}_{c(\rho,z'')} + K_{\rho}),
\end{align*}
where $c = c(\rho,z'') := \prod_{j = 1}^q \gamma_{a_j,k,\lambda}(\rho)c_j(z'') \in \mathrm{SO}(\mathbb{Z}_+^m) \otimes {\rm VO}_{\partial}(\mathbb{B}^{n-\ell})$,
\[K_{\rho}:= \left(\prod_{j = 1}^q \gamma_{a_j,k,\lambda}(\rho)\right)(T^{\lambda+|\rho|+\ell}_{c_1} \cdots T^{\lambda+|\rho|+\ell}_{c_q} - T^{\lambda+|\rho|+\ell}_{c(\rho,z'')})\]
is compact by the compact semi-commutator property of ${\rm VO}_{\partial}(\mathbb{B}^{n-\ell})$ and $\|K_{\rho}\| \to 0$ as $|\rho| \to \infty$ by \cite[Corollary 3.10]{BaHaVa}. 
Hence all finite sums of finite products of generators can be expressed in this form. The most difficult part now is to show that this remains true for operators in the closure.

Let $\boldsymbol{T}^{\lambda} \in \boldsymbol{\mathcal{T}}_{\lambda}(L^{\infty}_{k\textup{-}qr}, {\rm VO}_{\partial}(\mathbb{B}^{n-\ell}))$. Then there is a sequence of operators $(\boldsymbol{T}_q^{\lambda})_{q \in \N}$ that have the representation
\[\boldsymbol{T}_q^{\lambda} \asymp \bigoplus_{\rho \in \mathbb{Z}_+^m} I_{\mathscr{H}_{\rho}} \otimes (T_{c_q(\rho,z'')}^{\lambda+|\rho|+\ell} + K_{\rho,q})\]
and $\|\boldsymbol{T}^{\lambda} - \boldsymbol{T}_q^{\lambda}\| \to 0$ as $q \to \infty$.

By \eqref{tensor_product_maximal_ideals}, we may interpret $c_q$ as a continuous function on $M(\textup{SO}) \times M(\textup{VO})$.  
Let $\mu \in M_{\infty}$, and choose a net $(\rho_{\eta})$ in $\mathbb{Z}_+^m$ that converges to $\mu$. Clearly, $|\rho_{\eta}| \to \infty$ as $\rho_{\eta} \to \mu$. Now since
\[\left\|\mathcal{B}_{\lambda+|\rho|+\ell}[K_{\rho,q}]\right\|_{\infty} \leq \|K_{\rho,q}\| \to 0\]
as $|\rho| \to \infty$, we get
\begin{align*}
& \lim_{\rho_{\eta} \to \mu} \mathcal{B}_{\lambda+|\rho_{\eta}|+\ell}\big{[}T_{c_q(\rho_{\eta},\cdot)}^{\lambda+|\rho_{\eta}|+\ell} + K_{\rho_{\eta},q}\big{]}(z'') 
= \lim_{\rho_{\eta} \to \mu} \mathcal{B}_{\lambda+|\rho_{\eta}|+\ell}\big{[}T_{c_q(\rho_{\eta},\cdot)}^{\lambda+|\rho_{\eta}|+\ell}\big{]}(z'')\\
& \ = \lim_{\rho_{\eta} \to \mu} \mathcal{B}_{\lambda+|\rho_{\eta}|+\ell}\big{(}c_q(\rho_{\eta},\cdot)\big{)}(z'') 
= \lim_{|\rho_{\eta}| \to \infty} \mathcal{B}_{\lambda+|\rho_{\eta}|+\ell}(c_q(\mu,\cdot))(z'')\\
& \  + \lim_{\rho_{\eta} \to \mu} \mathcal{B}_{\lambda+|\rho_{\eta}|+\ell}\big{(}c_q(\rho_{\eta},\cdot) - c_q(\mu,\cdot)\big{)}(z'') = c_q(\mu,z'')
\end{align*}
uniformly for all $z'' \in \mathbb{B}^{n-\ell}$ by \cite[Proposition 4.4]{BC1}. As
\[\|\boldsymbol{T}_{q_1}^{\lambda} - \boldsymbol{T}_{q_2}^{\lambda}\| 
= \sup_{\rho \in \mathbb{Z}_+^m} \|T_{c_{q_1}(\rho,\cdot)}^{\lambda+|\rho|+\ell} + K_{\rho,q_1} - 
T_{c_{q_2}(\rho,\cdot)}^{\lambda+|\rho|+\ell} + K_{\rho,q_2}\|,\]
this implies $\|c_{q_1}(\mu,\cdot) - c_{q_2}(\mu,\cdot)\|_{\infty} \leq \|\boldsymbol{T}_{q_1}^{\lambda} - \boldsymbol{T}_{q_2}^{\lambda}\|$. Therefore the sequence $(c_q)_{q \in \N}$ 
restricted to $\mu \in M_{\infty}$ is a Cauchy sequence and hence converges uniformly to some function 
$$c_{\infty} \in  C\big{(}M_{\infty} \times  M(\textup{VO})\big{)}.$$
\par 
Similarly, let $\nu \in  M_{\partial}$ and choose a net $(z''_{\eta})$ in $\mathbb{B}^{n-\ell}$ that converges to $\nu$. 
Clearly, $|z''_{\eta}| \to 1$ as $z''_{\eta} \to \nu$. By \cite[Theorem 5.5]{MiSuWi}, we have \\ $\mathcal{B}_{\lambda+|\rho|+\ell}[K_{\rho,q}](z''_{\eta}) \to 0$ for every 
fixed $\rho \in \mathbb{Z}_+^m$ as $|z''_{\eta}| \to 1$. Thus
\begin{align*}
\lim_{z''_{\eta} \to \nu} \mathcal{B}_{\lambda+|\rho|+\ell}\big{[}T_{c_q(\rho,\cdot)}^{\lambda+|\rho|+\ell} + K_{\rho,q}\big{]}(z''_{\eta}) 
&= \lim_{z''_{\eta} \to \nu} \mathcal{B}_{\lambda+|\rho|+\ell}[T_{c_q(\rho,\cdot)}^{\lambda+|\rho|+\ell}](z''_{\eta})\\
&= \lim_{z''_{\eta} \to \nu} \mathcal{B}_{\lambda+|\rho|+\ell}\big{(}c_q(\rho,\cdot)\big{)}(z''_{\eta}) = c_q(\rho,\nu)
\end{align*}
for every $\rho \in \mathbb{Z}_+^m$ by Proposition \ref{prop:Berezin_VO}. As above, it follows that $(c_q)_{q \in \N}$ restricted to 
$\nu \in M_{\partial}$ uniformly converges to some $c_{\partial} \in C(M(\textup{SO})  \times M_{\partial})$. 
Hence we may define a continuous function $c$ as follows:
\[c(\mu,\nu) := \begin{cases} c_{\infty}(\mu,\nu) & \text{for } (\mu,\nu) \in M_{\infty} \times M(\textup{VO}), \\ 
c_{\partial} (\mu,\nu) & \text{for } (\mu,\nu) \in M(\textup{SO})  \times M_{\partial}. 
\end{cases}\]
Let
\begin{equation} \label{diamond}
\Diamond := \left( M_{\infty} \times    M(\textup{VO}) \right) \cup
 \left(M(\textup{SO})  \times M_{\partial} \right),
 \end{equation}
i.e.~the domain of $c$. As $\Diamond$ is a closed subset of $M(\textup{SO})  \times M(\textup{VO})$, we may extend $c$ to all of 
$M(\textup{SO})  \times M(\textup{VO})$ without increasing its norm by {\it Tietze's theorem}. Similarly, we may define the remainders 
$d_q(\mu,\nu) := c(\mu,\nu) - c_q(\mu,\nu)$ on $\Diamond$ and extend them to $M(\textup{SO}) \times M(\textup{VO})$ 
without increasing their norms. In particular, $d_q$ converges uniformly to $0$ as $q \to \infty$. Now consider $c_q' := c - d_q$ which coincides with $c_q$ on $\Diamond$ and therefore $c_q'-c_q \in c_0 \otimes C_0(\mathbb{B}^{n-\ell})$. By Corollary \ref{co:comp}, this implies that $T_{c_q'(\rho,z'')-c_q(\rho,z'')}^{\lambda+|\rho|+\ell}$ is compact. We may write
\begin{eqnarray*}
 \boldsymbol{T}_q^{\lambda} &\asymp& \bigoplus_{\rho \in \mathbb{Z}_+^m} I_{\mathscr{H}_{\rho}} \otimes \big{(}T_{c_q(\rho,z'')}^{\lambda+|\rho|+\ell} + K_{\rho,q}\big{)} \\
&=& \bigoplus_{\rho \in \mathbb{Z}_+^m} I_{\mathscr{H}_{\rho}} \otimes \big{(}T_{c_q'(\rho,z'')}^{\lambda+|\rho|+\ell} + T_{c_q(\rho,z'')-c_q^{\prime}(\rho,z'')}^{\lambda+|\rho|+\ell} + K_{\rho,q}\big{)},
\end{eqnarray*}
where $T_{c_q'(\rho,z'')}^{\lambda+|\rho|+\ell}$ converges to $T_{c(\rho,z'')}^{\lambda+|\rho|+\ell}$ (uniformly in $\rho$) and $T_{c_q(\rho,z'')-c_q^{\prime}(\rho,z'')}^{\lambda+|\rho|+\ell} + K_{\rho,q}$ is compact and tends to $0$ as $|\rho| \to \infty$. Taking the limit $q \to \infty$ implies that $\boldsymbol{T}^{\lambda}$ may be written as
\[\boldsymbol{T}^{\lambda} \asymp \bigoplus_{\rho \in \mathbb{Z}_+^m} I_{\mathscr{H}_{\rho}} \otimes (T_{c(\rho,z'')}^{\lambda+|\rho|+\ell} + K_{\rho}),\]
where $K_{\rho}$ is a compact operator. It remains to show that $\|K_{\rho}\| \to 0$ as $|\rho| \to \infty$. Choose $q$ sufficiently large such that 
$\|K_{\rho} - T_{c_q'(\rho,z'')-c_q(\rho,z'')}^{\lambda+|\rho|+\ell} - K_{\rho,q}\| < \varepsilon$ for all $\rho \in \mathbb{Z}_+^m$. Now choose $|\rho|$ sufficiently large such that 
$\|T_{c_q'(\rho,z'')-c_q(\rho,z'')}^{\lambda+|\rho|+\ell} - K_{\rho,q}\| < \varepsilon$. This implies $\|K_{\rho}\| < 2\varepsilon$ and the conclusion follows.
\end{proof}
We remark that, contrary to the statement of Theorem \ref{th:rep_VO}, the representation (\ref{eq:big-repr}) of an operator $\boldsymbol{T}^{\lambda} \in \boldsymbol{\mathcal{T}}_{\lambda}(L^{\infty}_{k\textup{-}qr}, {\rm VO}_{\partial}(\mathbb{B}^{n-\ell}))$ is not unique. The source of such a non-uniqueness is due to the following non-unique representation of the zero operator (Corollary \ref{co:comp})
\begin{equation} \label{eq:0-op}
  0 \asymp \bigoplus_{\rho \in \mathbb{Z}_+^m} I_{\mathscr{H}_{\rho}} \otimes (T_{d(\rho,z'')}^{\lambda+|\rho|+\ell} + K_{\rho}),
\end{equation}
where $d=d(\rho,z'') \in c_0 \otimes C_0(\mathbb{B}^{n-\ell})$ and $K_{\rho} = - T_{d(\rho,z'')}^{\lambda+|\rho|+\ell}$.
\vspace{1mm}\par 
We describe now a procedure of how to recover a representation \eqref{eq:big-repr} for an operator $\boldsymbol{T}^{\lambda} \in 
\boldsymbol{\mathcal{T}}_{\lambda}(L^{\infty}_{k\textup{-}qr}, {\rm VO}_{\partial}(\mathbb{B}^{n-\ell}))$. Decomposing $\boldsymbol{T}^{\lambda}$ into levels, we get
\[\boldsymbol{T}^{\lambda} \asymp \bigoplus_{\rho \in \mathbb{Z}_+^m} I_{\mathscr{H}_{\rho}} \otimes T_{\rho}\]
for some operators $T_{\rho}$ (cf.~Equation \eqref{quasi-radial_VO_decomposition}). According to Theorem \ref{th:big-repr}, there is a function 
$c = c(\rho,z'') \in \mathrm{SO}(\mathbb{Z}_+^m) \otimes {\rm VO}_{\partial}(\mathbb{B}^{n-\ell})$ such that
\[T_{\rho} = T_{c(\rho,\cdot)}^{\lambda+|\rho|+\ell} + K_{\rho}\]
for some compact operator $K_{\rho}$. This function $c$ is actually unique modulo $c_0 \otimes C_0(\mathbb{B}^{n-\ell})$ (which just produces different compact operators 
$K_{\rho}$) and can be derived as follows:
\vspace{1mm}\par 
Let $\mu \in M(\textup{SO})  \setminus \mathbb{Z}_+^m$ and choose a net $(\rho_{\eta})$ in $\mathbb{Z}_+^m$ that converges to $\mu$. Then, as in the proof of Theorem \ref{th:big-repr}, we have
\begin{align*}
c(\mu,z'') &= \lim_{\rho_{\eta} \to \mu} \mathcal{B}_{\lambda+|\rho_{\eta}|+\ell}\big{(}c(\rho_{\eta},\cdot)\big{)}(z'') 
= \lim_{\rho_{\eta} \to \mu} \mathcal{B}_{\lambda+|\rho_{\eta}|+\ell}\big{[}T_{c(\rho_{\eta},\cdot)}^{\lambda+|\rho_{\eta}|+\ell}\big{]}(z'')\\
&= \lim_{\rho_{\eta} \to \mu} \mathcal{B}_{\lambda+|\rho_{\eta}|+\ell}\big{[}T_{c(\rho_{\eta},\cdot)}^{\lambda+|\rho_{\eta}|+\ell} + K_{\rho_{\eta}}\big{]}(z'') 
= \lim_{\rho_{\eta} \to \mu} \mathcal{B}_{\lambda+|\rho_{\eta}|+\ell}\big{[}T_{\rho_{\eta}}\big{]}(z'')
\end{align*}
for every $z'' \in \mathbb{B}^{n-\ell}$. Likewise, let $\nu \in M_{\partial}$ and choose a net $(z''_{\eta})$ in $\mathbb{B}^{n-\ell}$ that converges to $\nu$. Then
\begin{align*}
c(\rho,\nu) &= \lim_{z''_{\eta} \to \nu} \mathcal{B}_{\lambda+|\rho|+\ell}\big{(}c(\rho,\cdot)\big{)}(z''_{\eta})
= \lim_{z''_{\eta} \to \nu} \mathcal{B}_{\lambda+|\rho|+\ell}\big{[}T_{c(\rho,\cdot)}^{\lambda+|\rho_{\eta}|+\ell}\big{]}(z''_{\eta})\\
&= \lim_{z''_{\eta} \to \nu} \mathcal{B}_{\lambda+|\rho|+\ell}\big{[}T_{c(\rho,\cdot)}^{\lambda+|\rho|+\ell} + K_{\rho}\big{]}(z''_{\eta}) 
= \lim_{z''_{\eta} \to \nu} \mathcal{B}_{\lambda+|\rho|+\ell}\big{[}T_{\rho}\big{]}(z''_{\eta})
\end{align*}
for $\rho \in \mathbb{Z}_+^m$. This fixes the symbol $c$ on the set $\Diamond$ defined in (\ref{diamond}).  
By applying the same arguments as in the proof of Theorem \ref{th:big-repr} it follows that $c$ is uniquely determined up to a function in $c_0 \otimes C_0(\mathbb{B}^{n-\ell})$. 
The compact part is now of course just $T_{\rho} - T_{c(\rho,\cdot)}^{\lambda+|\rho|+\ell}$, which  then depends on the actual choice of $c$.
\vspace{1mm}\par 
We wish to describe all irreducible representations of the algebra $\boldsymbol{\mathcal{T}}_{\lambda}(L^{\infty}_{k\textup{-}qr}, {\rm VO}_{\partial}(\mathbb{B}^{n-\ell}))$. The direct sum decomposition (\ref{eq:big-repr}) of its elements induces the irreducible representations:

\begin{itemize}
 \item [\textup{(i)}] {\it infinite dimensional (identical) representations} $\iota_{\rho}$ on $\mathcal{A}^2_{\lambda+|\rho|+\ell}(\mathbb{B}^{n-\ell})$, with $\rho \in \mathbb{Z}_+^m$, defined  by
\begin{equation*}
  \iota_{\rho}\, : \ \boldsymbol{T}^{\lambda} \asymp \bigoplus_{\rho \in \mathbb{Z}_+^m} I_{\mathscr{H}_{\rho}} \otimes (T_{c(\rho,z'')}^{\lambda+|\rho|+\ell} + K_{\rho}) \ \ \longmapsto \ \ T_{c(\rho,z'')}^{\lambda+|\rho|+\ell} + K_{\rho};
\end{equation*}
  \item [\textup{(ii)}] {\it one-dimensional representations} $\pi_{\rho, \nu}$ with $(\rho,\nu) \in \mathbb{Z}_+^m \times  M_{\partial} $, defined  by
\begin{equation*}
 \pi_{\rho,\nu}\, : \ \boldsymbol{T}^{\lambda} \asymp \bigoplus_{\rho \in \mathbb{Z}_+^m} I_{\mathscr{H}_{\rho}} \otimes (T_{c(\rho,z'')}^{\lambda+|\rho|+\ell} + K_{\rho}) \ \ \longmapsto \ \ c(\rho,\nu).
\end{equation*}
\end{itemize}
It is straightforward to see that these representations are not pairwise unitary equivalent. The ambiguity of the form (\ref{eq:big-repr}), caused by (\ref{eq:0-op}), does not effect 
the action of the above representations. We list representations which are induced by the ``quantization effect'': 
\vspace{1mm}\par
Given any $\mu \in M_{\infty}$, let $\{\rho_{\beta}\}_{\beta \in B}$ be a net converging to $\mu$. Then the map
\begin{eqnarray*}
 \nu_{\mu} &:&  \boldsymbol{T}^{\lambda} \asymp \bigoplus_{\rho \in \mathbb{Z}_+^m} I_{\mathscr{H}_{\rho}} \otimes (T_{c(\rho,z'')}^{\lambda+|\rho|+\ell} + K_{\rho}) \ 
 \longmapsto \ \lim_{\beta} \mathcal{B}_{\lambda+|\rho_{\beta}|+\ell} \big{[}T_{c(\rho_{\beta},z'')}^{\lambda+|\rho_{\beta}|+\ell} +  K_{\rho_{\beta}}\big{]} \\
&& = \ \lim_{\beta} \mathcal{B}_{\lambda+|\rho_{\beta}|+\ell} \big{[}T_{c(\rho_{\beta},z'')}^{\lambda+|\rho_{\beta}|+\ell}\big{]} =  c(\mu,\cdot)
\end{eqnarray*}
is well defined and a $^*$-homomorphism of the $C^*$-algebra $\boldsymbol{\mathcal{T}}_{\lambda}(L^{\infty}_{k\textup{-}qr}, {\rm VO}_ {\partial}(\mathbb{B}^{n-\ell}))$ onto 
the algebra $C(M(\textup{VO}))$. Note that the ambiguity of the form (\ref{eq:big-repr}), caused by (\ref{eq:0-op}), does not effect the action of this map.
\vspace{1mm}\par 
Each such map $\nu_{\mu}$ induces a family of one-dimensional representations of the $C^*$-algebra $\boldsymbol{\mathcal{T}}_{\lambda}(L^{\infty}_{k\textup{-}qr}, {\rm VO}_ {\partial}(\mathbb{B}^{n-\ell}))$, 
defined for each  $(\mu,z'') \in M_{\infty} \times M(\textup{VO})$ as follows:
\begin{eqnarray*}
 \nu_{\mu,z''} &:& \boldsymbol{T}^{\lambda} \asymp \bigoplus_{\rho \in \mathbb{Z}_+^m} I_{\mathscr{H}_{\rho}} \otimes (T_{c(\rho,z'')}^{\lambda+|\rho|+\ell} + K_{\rho}) \ 
 \overset{\nu_{\mu}}{\longmapsto} \ c(\mu,\cdot) \in C\big{(}M(\textup{VO})\big{)} 
 \\   && 
 \ \mapsto \ c(\mu,z'') \in \mathbb{C}.
\end{eqnarray*}
Consider operators  $\boldsymbol{T}_j^{\lambda} = \boldsymbol{T}_{c_j}^{\lambda} + \boldsymbol{K}_j^{\lambda}$ with $c_j \in \mathrm{SO}(\mathbb{Z}_+^m) \otimes {\rm VO}_{\partial}(\mathbb{B}^{n-\ell})$ and $\boldsymbol{K}_j^{\lambda}$ compact for $j=1,2$. Then the difference $\boldsymbol{T}_1^{\lambda} -\boldsymbol{T}_2^{\lambda}$ 
is compact if and only if $ \boldsymbol{T}_{c_1-c_2}^{\lambda}$ is compact which is equivalent to $$c_1-c_2 \in c_0 \times C_0(\mathbb{B}^{n-\ell})$$ (see Corollary \ref{co:comp}). 
As a consequence the Calkin algebra $\widehat{\boldsymbol{\mathcal{T}}}_{\lambda}\big{(}L^{\infty}_{k\textup{-}qr}, {\rm VO}_ {\partial}(\mathbb{B}^{n-\ell})\big{)}$, defined as
$$
\boldsymbol{\mathcal{T}}_{\lambda}\big{(}L^{\infty}_{k\textup{-}qr}, {\rm VO}_ {\partial}(\mathbb{B}^{n-\ell})\big{)} 
/ \left(\boldsymbol{\mathcal{T}}_{\lambda}\big{(}L^{\infty}_{k\textup{-}qr}, {\rm VO}_ {\partial}(\mathbb{B}^{n-\ell})\big{)} \cap 
\boldsymbol{\mathcal{K}}(\boldsymbol{\mathcal{A}}^2_{\lambda}(\mathbb{B}^n))\right),$$
is isomorphic and isometric to the quotient
\begin{equation*}
C\big{(} M(\textup{SO})   \times  M(\textup{VO}) \big{)} / \left(c_0 \times C_0(\mathbb{B}^{n-\ell})\right)
= C\big{(}M_{\infty} \times  M(\textup{VO}) \cup \mathbb{Z}_+^m \times M_{\partial} \big{)}.
\end{equation*}

\begin{thm}\label{Section_6_complete_list_of_representations}
The following list of irreducible representations of \\ $\boldsymbol{\mathcal{T}}_{\lambda}(L^{\infty}_{k\textup{-}qr}, {\rm VO}_ {\partial}(\mathbb{B}^{n-\ell}))$ is complete up to unitary equivalence:
\begin{itemize}
	\item[\textup{(i)}] $\iota_{\rho}\, : \ \boldsymbol{T}^{\lambda} \asymp \bigoplus_{\rho \in \mathbb{Z}_+^m} I_{\mathscr{H}_{\rho}} \otimes (T_{c(\rho,z'')}^{\lambda+|\rho|+\ell} + K_{\rho}) \ 
	\longmapsto \ T_{c(\rho,z'')}^{\lambda+|\rho|+\ell} + K_{\rho}$ for $\rho \in \mathbb{Z}_+^m$,
	\item[\textup{(ii)}] $\pi_{\rho,\nu}\, : \ \boldsymbol{T}^{\lambda} \asymp \bigoplus_{\rho \in \mathbb{Z}_+^m} I_{\mathscr{H}_{\rho}} \otimes (T_{c(\rho,z'')}^{\lambda+|\rho|+\ell} + K_{\rho}) \ \longmapsto \ c(\rho,\nu)$ 
	for $(\rho,\nu) \in \mathbb{Z}_+^m \times M_{\partial}$,
	\item[\textup{(iii)}] $\nu_{\mu,z''}\, : \ \boldsymbol{T}^{\lambda} \asymp \bigoplus_{\rho \in \mathbb{Z}_+^m} I_{\mathscr{H}_{\rho}} \otimes (T_{c(\rho,z'')}^{\lambda+|\rho|+\ell} + K_{\rho}) \longmapsto c(\mu,z'')$ for $(\mu,z'') \in  M_{\infty} \times M(\textup{VO})$.
\end{itemize}
Moreover, the above representations are pairwise not unitarily equivalent.
\end{thm}
\begin{proof}
By the discussion above, it remains to show that the list is complete. We may proceed as in the proof of Theorem \ref{th:rep_list}. By \cite[Proposition 2.11.2]{Di}, every irreducible representation of 
$\boldsymbol{\mathcal{T}}_{\lambda}(L^{\infty}_{k\textup{-}qr},{\rm VO}_{\partial}(\mathbb{B}^{n-\ell}))$ is either induced by a representation of the quotient
\begin{equation}\label{Quotient_Main_theorem_Section_6}
\boldsymbol{\mathcal{T}}_{\lambda}\big{(}L^{\infty}_{k\textup{-}qr},{\rm VO}_{\partial}(\mathbb{B}^{n-\ell})\big{)} / \big{(}\boldsymbol{\mathcal{T}}_{\lambda}\big{(}L^{\infty}_{k\textup{-}qr},{\rm VO}_{\partial}(\mathbb{B}^{n-\ell})) \cap \boldsymbol{\mathcal{K}}(\mathcal{A}_{\lambda}^2(\mathbb{B}^n))\big{)}
\end{equation}
or extends an irreducible representation of 
\begin{equation}\label{Ideal_Main_theorem_Section_6}
\boldsymbol{\mathcal{T}}_{\lambda}(L^{\infty}_{k\textup{-}qr},{\rm VO}_{\partial}(\mathbb{B}^{n-\ell})) \cap \boldsymbol{\mathcal{K}}(\mathcal{A}_{\lambda}^2(\mathbb{B}^n)). 
\end{equation}
Since by our previous remark (\ref{Quotient_Main_theorem_Section_6}) is isometrically isomorphic to 
\begin{equation*}
C\big{(}M_{\infty} \times  M(\textup{VO}) \cup \mathbb{Z}_+^m \times M_{\partial} \big{)},
\end{equation*}
the former representations are exactly $\nu_{\mu,z''}$ and $\pi_{\rho,\nu}$, respectively. It follows from Lemma \ref{le:comp_k-qr_VO} that every irreducible representation of 
(\ref{Ideal_Main_theorem_Section_6})  is induced by an irreducible representation of $\mathcal{K}(\mathcal{A}^2_{\lambda+|\rho|+\ell}(\mathbb{B}^{n-\ell}))$ for some $\rho \in \mathbb{Z}_+^m$ on the respective level. As the only irreducible representation of $\mathcal{K}(\mathcal{A}^2_{\lambda+|\rho|+\ell}(\mathbb{B}^{n-\ell}))$ is the identical representation, we conclude that the representations $\iota_{\rho}$ are the only additional representations we get.
\end{proof}

Similarly to the Fredholm characterization of Toeplitz operators at the end of Section~\ref{se:VO}, we have the following result.

\begin{prop}
 Let $c(\rho, z'') \in \left(\mathrm{SO}(\mathbb{Z}_+^m) \otimes {\rm VO}_{\partial}(\mathbb{B}^{n-\ell})\right) \otimes \mathrm{Mat}_p(\mathbb{C})$. Then the Toeplitz operator
\begin{equation*}
 \mathbf{T}^{\lambda} \asymp \bigoplus_{\rho \in \mathbb{Z}_+^m} I \otimes (T_c^{\lambda+|\rho|+\ell} + K_{\rho}) \in \boldsymbol{\mathcal{T}}_{\lambda}(L^{\infty}_{k\textup{-}qr},{\rm VO}_{\partial}(\mathbb{B}^{n-\ell})) \otimes \mathrm{Mat}_p(\mathbb{C})
\end{equation*}
is Fredholm if and only if the restriction of the matrix $c(\eta)$ onto $(M_{\infty} \times M(\textup{VO})) \cup (\mathbb{Z}_+^m \times  M_{\partial})$ is invertible. 
The essential spectrum of the operator $\mathbf{T}^{\lambda}$ is given by
\begin{equation*}
 \esssp \mathbf{T}^{\lambda} = \mathrm{Range}\,\det c|_{M_{\infty} \times   M(\textup{VO}) \cup \mathbb{Z}_+^m \times M_{\partial}}.
\end{equation*}
\end{prop}

Finally, we give an index formula for Fredholm operators in the algebra \[\boldsymbol{\mathcal{T}}_{\lambda}(L^{\infty}_{k\textup{-}qr},{\rm VO}_{\partial}(\mathbb{B}^{n-\ell})) \otimes \mathrm{Mat}_p(\mathbb{C}).\]
Let the operator 
\begin{equation*}
 \mathbf{T}^{\lambda} \asymp \bigoplus_{\rho \in \mathbb{Z}_+^m} I \otimes (T_c^{\lambda+|\rho|+\ell} + K_{\rho})
\end{equation*}
be Fredholm. Then, by \cite[Section 2]{XZ}, for each $\rho \in \mathbb{Z}_+^m$ there is $s_{\rho,0} \in (0,1)$ such that for every $s_{\rho} \in (s_{\rho,0},1)$ we have
\begin{equation*}
 \Ind (T_c^{\lambda+|\rho|+\ell} + K_{\rho}) = \Ind (T_{c_{s_{\rho}}}^{\lambda+|\rho|+\ell} + K_{\rho}),
\end{equation*}
where $c_{s_{\rho}} = c_{s_{\rho}}(\rho, \cdot) \in C(\overline{\mathbb{B}}^{n-\ell})\otimes \mathrm{Mat}_p(\mathbb{C})$ is the matrix-valued function of type \eqref{eq:c_s}. Thus we have the following index formula
\begin{eqnarray*}
 \Ind \mathbf{T}^{\lambda} &=& \sum_{\rho \in \mathbb{Z}_+^m} \dim \mathscr{H}_{\rho} \times \Ind (T_c^{\lambda+|\rho|+\ell} + K_{\rho}) \\
 &=& \sum_{\rho \in \mathbb{Z}_+^m} \dim \mathscr{H}_{\rho} \times \Ind (T_{c_{s_{\rho}}}^{\lambda+|\rho|+\ell} + K_{\rho}),
\end{eqnarray*}
where the indices of the operators in the last line can be calculated by {\rm \cite[Theorem 1.5]{Venugopalkrishna72}}. \\ Note that the above infinite sum has, in fact, only a finite number of non-zero summands.

\end{document}